\documentclass[11pt]{amsart}
\usepackage{palatino, mathpazo}
\usepackage{amsfonts}
\usepackage{amsmath}
\usepackage{amssymb,latexsym}
\usepackage{graphicx}
\usepackage{amssymb}
\usepackage[mathscr]{eucal}
\usepackage{color}
\providecommand{\U}[1]{\protect \rule{.1in}{.1in}}

\newtheorem{theorem}{Theorem}[section]

\newtheorem{corollary}[theorem]{Corollary}

\newtheorem{lemma}[theorem]{Lemma}

\theoremstyle{remark}
\newtheorem{remark}[theorem]{Remark}

\numberwithin{equation}{section}
\begin{document}
\title[the classical Eisenstein series of weight two]{Critical points of the classical Eisenstein series of weight two}
\author{Zhijie Chen}
\address{Department of Mathematical Sciences, Yau Mathematical Sciences Center,
Tsinghua University, Beijing, 100084, China }
\email{zjchen2016@tsinghua.edu.cn}
\author{Chang-Shou Lin}
\address{Taida Institute for Mathematical Sciences (TIMS), Center for Advanced Study in
Theoretical Sciences (CASTS), National Taiwan University, Taipei 10617, Taiwan }
\email{cslin@math.ntu.edu.tw}

\begin{abstract}
In this paper, we completely determine the critical points of the normalized Eisenstein series $E_2(\tau)$ of weight $2$. Although $E_2(\tau)$ is not a modular form, our result shows that $E_2(\tau)$ has at most one critical point in every fundamental domain of $\Gamma_{0}(2)$. We also give a criteria for a fundamental domain containing a critical point of $E_2(\tau)$.  Furthermore, under the M\"{o}bius transformation of $\Gamma_{0}(2)$ action, all critical points can be mapped into the basic fundamental domain $F_0$ and their images are contained densely on three smooth curves. A geometric interpretation of these smooth curves is also given. It turns out that these smooth curves coincide with the degeneracy curves of trivial critical points of a multiple Green function related to flat tori.
\end{abstract}
\maketitle
\tableofcontents

\section{Introduction}

The Jacobi theta functions, the Eisenstein series and the Weierstrass functions arise in numerous theories and applications of both mathematics and physics. Since their discovery in the early 19th century, the mathematical foundation of elliptic functions was subsequently developed. It turns out that, besides their applications in science, these special functions in the elliptic function theory are rather deep objects by themselves.

The main goal of this paper is to completely locate all the critical points of the classical function $\eta_1(\tau)$ or equivalently the normalized Eisenstein series $E_2(\tau)$ of weight $2$.
Throughout the paper, we use the notations $\mathbb{R}^{+}=(0,+\infty)$,  $\omega_{1}=1$,
$\omega_{2}=\tau$, $\omega_{3}=1+\tau$ and $\Lambda_{\tau}=\mathbb{Z+Z}\tau$,
where $\tau \in \mathbb{H}=\{ \tau|\operatorname{Im}\tau>0\}$.
Let $\wp(z)=\wp(z|\tau)$ be the Weierstrass  $\wp$-function with periods $\Lambda_{\tau}$, defined by
\[
\wp(z|\tau):=\frac{1}{z^{2}}+\sum_{\omega \in \Lambda_{\tau}\backslash
\{0\}}\left(  \frac{1}{(z-\omega)^{2}}-\frac{1}{\omega^{2}}\right).
\]
Let $\zeta(z)=\zeta(z|\tau):=-\int^{z}\wp(\xi|\tau)d\xi$ be the Weierstrass zeta
function, which is odd and has two quasi-periods $\eta_{k}(\tau):=2\zeta(\frac{\omega_k}{2}|\tau)$, $k=1,2$:
\begin{equation}
\eta_{1}(\tau)=\zeta(z+1|\tau)-\zeta(z|\tau),\text{ \ }\eta_{2}(\tau
)=\zeta(z+\tau|\tau)-\zeta(z|\tau). \label{quasi}%
\end{equation}
The well-known Legendre relation gives $\eta_2(\tau)=\tau\eta_{1}(\tau)-2\pi i$. In the literature, $\eta_1(\tau)$ is known as the Weierstrass eta function (cf. \cite{YB}), which is just a multiple of the normalized Eisenstein series $E_2(\tau)$ of weight $2$:
\begin{align}\label{eisenstein}\frac{3}{\pi^2}\eta_1(\tau)=E_2(\tau)
:=&\frac{3}{\pi^2}\sum_{m=-\infty}^{\infty}\mathop{{\sum}'}_{n=-\infty}^{\infty}\frac{1}{(m\tau+n)^2}\\
=&1-24\sum_{n=1}^{\infty}b_{n} e^{2n\pi i \tau},\quad b_n=\sum_{1\leq d|n}d.\nonumber\end{align}
Conventionally, $\mathop{{\sum}'}$ means to sum over $(n,m)\in \mathbb{Z}^2\setminus\{(0,0)\}$. Besides, $\eta_1(\tau)$ is also connected with Dedekind eta function
\[\eta(\tau):=e^{\frac{\pi i\tau}{12}}\prod_{n=1}^\infty (1-e^{2n\pi i\tau})\]
through the following logarithmic differential formula (cf. \cite[p.696]{YB}):
\[\frac{1}{\eta(\tau)}\eta'(\tau)=\frac{i}{4\pi}\eta_{1}(\tau).\]

Unlike the other Eisenstein series of weight $2k$ with $k\geq 2$, $E_2(\tau)$ is not a modular form. Its transformation under the action of $SL(2,\mathbb{Z})$ satisfies
\begin{equation}\label{sl-eta}
E_{2}\left(\frac{a\tau+b}{c\tau+d}\right)=(c\tau+d)^{2}E_2(\tau)-\frac{6 ic}{\pi}(c\tau+d),\; \begin{pmatrix}
a & b\\
c & d
\end{pmatrix}
\in SL(2,\mathbb{Z}).
\end{equation}
Thus it is surprising that its critical points possess the following property.
\begin{theorem}\label{eta-thm0}
Let $F$ be a fundamental domain of $\Gamma_{0}(2)$. Then $E_{2}(\tau)$ has at most one critical point in $F$.
\end{theorem}
Here $\Gamma_{0}(2)$ is the congruence subgroup of $SL(2,\mathbb{Z})$ defined by
\[
\Gamma_{0}(2):=\left \{  \left.
\begin{pmatrix}
a & b\\
c & d
\end{pmatrix}
\in SL(2,\mathbb{Z})\right \vert c\equiv0\text{ }\operatorname{mod}2\right \}.
\]
Recently, there are some works studying the zeros of $E_{2}(\tau)$;  see \cite{EIBS,WY} and references therein. As far as we know, there seems no results concerning the critical points of $E_{2}(\tau)$ in the literature.

In view of Theorem \ref{eta-thm0}, a natural question is: \emph{What are those fundamental domains containing critical points}? To answer this question, we introduce our "basic" fundamental domain $F_{0}$ of $\Gamma
_{0}(2)$:
\[
F_{0}:=\{ \tau \in \mathbb{H}\ |\ 0\leqslant \  \text{Re}\  \tau \leqslant
1\  \text{and}\ |z-\tfrac{1}{2}|\geqslant \tfrac{1}{2}\}.
\]
Note
that given $\gamma=%
\begin{pmatrix}
a & b\\
c & d
\end{pmatrix}
\in \Gamma_{0}(2)/\{ \pm I_{2}\}$ (i.e. {\it consider $\gamma$ and $-\gamma$ to be
the same}),%
\[
\gamma (F_{0}):=\left \{  \left.  \gamma \cdot \tau:=\tfrac{a\tau+b}{c\tau
+d}\right \vert \tau \in F_{0}\right \}  =(-\gamma)(F_{0})%
\]
is another fundamental domain of $\Gamma_{0}(2)$. Moreover, $\gamma ( F_{0})=F_{0}+m$ for some $m\in \mathbb{Z}$ if and only if $c=0$.

\begin{theorem}\label{F-one} Let $F=\gamma (F_{0})$ be a fundamental domain of $\Gamma_{0}(2)$ with $\gamma=%
\begin{pmatrix}
a & b\\
c & d
\end{pmatrix}
\in \Gamma_{0}(2)/\{ \pm I_{2}\}$. Then $F$ contains a critical point of $E_{2}(\tau)$ if and only if $c\neq 0$.
\end{theorem}

By Theorem \ref{F-one}, we can transform every critical point of $E_{2}(\tau)$ via the M\"{o}bius transformation of $\Gamma_{0}(2)$ action to locate it in $F_0$. Denote \emph{the collection of such corresponding points in $F_0$ by $\mathcal{C}$}, which consists of infinitely many points. A fundamental question is: {\it What is the geometry of the set $\mathcal{C}$}?

Surprisingly, it turns out that $\mathcal{C}$ will locate on
{\it three smooth curves} $\tau(C)$ in $F_{0}$, which are parameterized by $C\in \mathbb{R}\setminus\{0,1\}$ via the following identity
\begin{equation}\label{c=tau}
C=\tau-\frac{2\pi i}{\eta_{1}(\tau)\pm \sqrt{g_{2}(\tau)/12}},\quad \tau\in F_{0}.
\end{equation}
Here $g_2(\tau)=60G_4(\tau)$ is the well-known invariant coming from
\[\wp'(z|\tau)^2
=4\wp(z|\tau)^3-g_{2}(\tau)\wp(z|\tau)-g_{3}(\tau),\]
and $G_4(\tau)$ is the Eisenstein series of weight $4$.
We will prove in Section 2 that for each $C\in \mathbb{R}\setminus\{0,1\}$, there is {\it a unique} $\tau(C)\in F_0$ such that (\ref{c=tau}) holds.\footnote{Note that the RHS of (\ref{c=tau}) is actually a multi-valued function, please see Theorem \ref{Unique-pole} for the precise definition of this unique $\tau(C)$.} Consequently, the parametrization (\ref{c=tau}) will give three smooth curves
\[
\mathcal{C}_{0}:=\{ \tau(C)|C\in(0,1)\},
\]
\begin{align*}
\mathcal{C}_{-}:=\{ \tau(C)|C\in(-\infty,0)\},&\text{ \ }\mathcal{C}_{+}:=\{
\tau(C)|C\in(1,+\infty)\}.
\end{align*}
The relation between (\ref{c=tau}) and $\eta_1'(\tau)$ comes from the classical formula (see e.g. \cite{YB}, or from Ramanujan's formula: $E_2'(\tau)=\frac{\pi i}{6}(E_2^2-E_4)$)
\begin{equation}\label{d-eta11}
\eta_{1}^{\prime}(\tau)=\tfrac{i}{2\pi}\left(  \eta_{1}(\tau)^{2}-\tfrac{1}
{12}g_{2}(\tau)\right).
\end{equation}

\begin{theorem}\label{Location} Let $\tau(C)$ be defined by (\ref{c=tau}) for $C\in \mathbb{R}\backslash \{0,1\}$. Then
\begin{equation}\label{cf0}\mathcal{C}=\left \{  \left.  \tau(\tfrac{-d}{c})\right \vert
\begin{pmatrix}
a & b\\
c & d
\end{pmatrix}
\in \Gamma_{0}(2)/\{ \pm I_{2}\} \text{ with }c\not =0\right \}\subset \mathcal{C}_{-}\cup\mathcal{C}_{0}\cup\mathcal{C}_{+}.\end{equation}
Furthermore, the closure of $\mathcal{C}$ in $F_0$ is precisely the union of the three smooth curves:
\begin{equation}\label{cf1}\overline{\mathcal{C}}\cap F_0= \mathcal{C}_{-}\cup\mathcal{C}_{0}\cup\mathcal{C}_{+}.\end{equation}
\end{theorem}

\begin{remark} In fact, we will prove $\tau(C)\in \mathring{F}_{0}$, where $\mathring{F}_{0}=F_{0}\backslash \partial F_{0}$ denotes the
set of interior points of $F_{0}$.  Given $\gamma=%
\begin{pmatrix}
a & b\\
c & d
\end{pmatrix}
\in \Gamma_{0}(2)/\{ \pm I_{2}\}$ with $c\neq 0$, we will prove in Theorem \ref{Coro-1} that the unique critical point of $E_{2}(\tau)$ in $\gamma(F_0)$ is precisely $\frac{a\tau(\tfrac{-d}{c})+b}{c\tau(\tfrac{-d}{c})+d}\in \gamma(\mathring{F}_{0})$.
Given $\gamma_{j}=%
\begin{pmatrix}
a_{j} & b_{j}\\
c_{j} & d_{j}%
\end{pmatrix}
\in \Gamma_{0}(2)/\{ \pm I_{2}\}$ with $c_{j}\not =0$ such that $\gamma
_{1}\not =\pm \gamma_{2}$, we have $\gamma_{1}(\mathring{F}_{0})\cap
\gamma_{2}(\mathring{F}_{0})=\emptyset$ (note that $\gamma_{1}(\partial F_{0})\cap
\gamma_{2}(\partial F_{0})\neq\emptyset$ may happen) and so%
\[
\frac{a_{1}\tau(\tfrac{-d_{1}}{c_{1}})+b_{1}}{c_{1}\tau(\tfrac{-d_{1}}{c_{1}%
})+d_{1}}\not =\frac{a_{2}\tau(\tfrac{-d_{2}}{c_{2}})+b_{2}}{c_{2}\tau
(\tfrac{-d_{2}}{c_{2}})+d_{2}}.
\]
Therefore, the map from $\mathcal{C}$ to the set of critical points of $E_{2}(\tau)$ is one-to-one.
The above results completely locate all the critical points of the Eisenstein series $E_{2}(\tau)$ or equivalently $\eta_{1}(\tau)$. To the best of our knowledge, such fundamental results have not appeared in the literature and are new. We believe that they will have important applications. For example, we consider $\tau=\frac{1}{2}+ib$ with $b>0$. Then $\eta_{1}(\tau)\in \mathbb{R}$. In order to study the behavior of the Green function on rhombus tori, Wang and the second author \cite{LW} considered the monotone property of $\eta_{1}(\tau)$ and  their numerical computation \cite[Figure 2]{LW} suggests that $\eta_{1}$ should increase from $0$ to some $b_0$ and then decrease after $b_0$, but they can not prove this assertion in \cite{LW} because (\ref{eisenstein}) implies \[\frac{3}{\pi^2}\eta_1(\tfrac{1}{2}+ib)=1-24\sum_{n=1}^{\infty}(-1)^nb_{n} e^{-2n\pi b},\quad b_n=\sum_{1\leq d|n}d>0,\nonumber\]from which it seems difficult to obtain the monotone property shown in \cite[Figure 2]{LW}. Now this assertion is confirmed by the following corollary.
\end{remark}

\begin{corollary}\label{line1/2} There exists $b_0\in (\frac{5}{24},\frac{1}{2\sqrt{3}})$ such that $\eta
_{1}(\frac12+ib)$ is strictly increasing for $b\in (0, b_0)$ and strictly decreasing for $b\in (b_0,+\infty)$.
\end{corollary}

One of our motivations of studying critical points of $\eta_{1}(\tau)$ comes from
the Green function on flat tori. Let $E_{\tau}:=\mathbb{C}/(\mathbb{Z}+\mathbb{Z}\tau)$ be a flat torus and $G(z)=G(z;\tau)$
be the Green function on the torus $E_{\tau}$:%
\begin{equation*}
-\Delta G(z;\tau)=\delta_{0}-\frac{1}{\left \vert E_{\tau}\right \vert
}\text{ \ on }E_{\tau},\quad
\int_{E_{\tau}}G(z;\tau)=0,
\end{equation*}
where $\delta_{0}$ is the Dirac measure at $0$ and $\left \vert E_{\tau
}\right \vert $ is the area of the torus $E_{\tau}$. See \cite{LW} for a detailed study of $G(z;\tau)$.
In \cite{CLW,LW2,LW3}, Chai, Wang and the second author introduced a multiple Green function $G_n$, $n\in\mathbb{N}$. Geometrically, any critical point of $G_n$ is closely related to bubbling phenomenon of nonlinear partial differential equations with exponential nonlinearities in two dimension; see \cite{CLW, LW3} for typical examples. Thus, understanding the critical points of $G_n$ is important for applications.

For the case $n=2$,
the
multiple Green function $G_{2}$ is defined by
\begin{equation}
G_{2}(z_{1},z_{2};\tau):=G(z_{1}-z_{2};\tau)-2G(z_{1};\tau)-2G(z_{2};\tau),
\label{51100}%
\end{equation}
where $0\neq z_1\neq z_2\neq 0$. A critical point $(a_{1},a_{2})$ of $G_{2}$ satisfies%
\begin{equation*}
2\nabla G(a_{1};\tau)=\nabla G(a_{1}-a_{2};\tau),\text{ \ }2\nabla
G(a_{2};\tau)=\nabla G(a_{2}-a_{1};\tau).
\end{equation*}
Clearly if $(a_{1},a_{2})$ is a critical point then so does
$(a_{2},a_{1})$, and we consider such two critical points to be
\emph{the same one}. A critical point $(a_{1},a_{2})$ is
called a \emph{trivial critical point }if%
\[
\{a_{1},a_{2}\}=\{-a_{1},-a_{2}\} \text{ \ in \ }E_{\tau}\text{.}%
\]
Recall $\omega_{1}=1, \omega_{2}=\tau$ and $\omega_{3}=1+\tau$.
It is known \cite{LW3} that $G_{2}$ has only five trivial critical points
$\{(\frac{1}{2}\omega_{i},\frac{1}{2}\omega_{j})|i\not =j\}$ and
$\{(q_{\pm},-q_{\pm})|\wp(q_{\pm}|\tau)=\pm \sqrt{g_{2}(\tau)/12}\}$, and the Hassian at $(q_{\pm},-q_{\pm})$ is given by
\begin{align}
&\det D^{2}G_{2}(q_{\pm},-q_{\pm};\tau)\nonumber\\
=&\frac{3|g_{2}(\tau)|}{4\pi
^{4}\operatorname{Im}\tau}|\wp(q_{\pm}|\tau)+\eta_{1}(\tau)|^{2}\operatorname{Im}
\bigg(\tau-\frac{2\pi i}{\eta_{1}(\tau)\pm \sqrt{g_{2}(\tau)/12}}\bigg). \label{SZ-22}%
\end{align}
From here and (\ref{c=tau}), we will prove in Section \ref{geometricinter} that the three curves coincide with degeneracy curves of $G_2$. The Hassian at $(\frac{1}{2}\omega_{i},\frac{1}{2}\omega_{j})$ is related to the critical points of the classical function $e_k(\tau):=\wp(\frac{\omega_k}{2}|\tau)$, $\{i,j,k\}=\{1,2,3\}$. We will study the critical points of $e_k(\tau)$ in another paper.

Our proof of the existence and uniqueness of $\tau(C)$ relies on a {\it pre-modular form} $Z_{r,s}^{(2)}(\tau)$ of weight $3$ introduced in \cite{LW2}. See also \cite{Dahmen}. For each pair $(r,s)\in \mathbb{R}^2\setminus \frac{1}{2}\mathbb{Z}^2$, $Z_{r,s}^{(2)}(\tau)$ is defined by
\[
Z_{r,s}^{(2)}(\tau):=Z_{r,s}(\tau)^{3}-3\wp(r+s\tau|\tau)Z_{r,s}(\tau
)-\wp^{\prime}(r+s\tau|\tau),
\]
where $Z_{r,s}(\tau)$ is introduced by Hecke \cite{Hecke}:
\begin{align}
Z_{r,s}(\tau)  :=\zeta(r+s\tau|\tau)-r\eta_{1}(\tau)-s\eta_{2}%
(\tau). \label{c-70}%
\end{align}
Indeed, if $(r,s)\in \mathbb{Q}^2$, $Z_{r,s}(\tau)$ is the well-known Eisenstein series of weight $1$ with characteristic $(r,s)$; see \cite[p.139]{Diamond-Shurman}. It is not difficult to see that $Z_{r,s}(\tau)$ is a modular form of weight $1$ with respect to $\Gamma(N)$ if $(r,s)$ is a $N$-torsion point, so $Z_{r,s}^{(2)}(\tau)$ is a modular form of weight $3$. See Section \ref{modularform}. The importance of $Z_{r,s}^{(2)}(\tau)$ lies on the fact that at any zero $\tau_0$ of $Z_{r,s}^{(2)}(\cdot)$, the pair $(r,s)$ contains all the monodromy data of the
classical Lam\'{e} equation
\begin{equation}\label{lame}
y''(z)=[n(n+1)\wp(z|\tau_0)+B]y(z),\quad n=2
\end{equation}
for some $B\in\mathbb{C}$; see \cite[Theorem 4.3]{LW2}. Therefore, it is important to study the zero of $Z_{r,s}^{(2)}(\cdot)$, which has not been settled yet. In this paper, we study the zero structure of $Z_{r,s}^{(2)}(\cdot)$. Define four open triangles (see Figure 1 in Section \ref{modularform}):
\begin{equation}%
\begin{array}
[c]{l}%
\triangle_{0}:=\{(r,s)\mid0<r,s<\tfrac{1}{2},\text{ }r+s>\tfrac{1}{2}\},\\
\triangle_{1}:=\{(r,s)\mid \tfrac{1}{2}<r<1,\text{ }0<s<\tfrac{1}{2},\text{
}r+s>1\},\\
\triangle_{2}:=\{(r,s)\mid \tfrac{1}{2}<r<1,\text{ }0<s<\tfrac{1}{2},\text{
}r+s<1\},\\
\triangle_{3}:=\{(r,s)\mid r>0,\text{ }s>0,\text{ }r+s<\tfrac{1}{2}\}.
\end{array}
\label{rectangle}%
\end{equation}

\begin{theorem}
\label{thm2}Let $(r,s)\in \lbrack0,1]\times \lbrack0,\frac{1}{2}]\backslash
\frac{1}{2}\mathbb{Z}^{2}$. Then $Z_{r,s}^{(2)}(\tau)=0$ has a solution $\tau$
in $F_{0}$ if and only if $(r,s)\in \triangle_{1}\cup \triangle_{2}\cup
\triangle_{3}$. Furthermore, for any $(r,s)\in \triangle_{1}\cup \triangle
_{2}\cup \triangle_{3}$, the zero $\tau \in F_{0}$ is unqiue and satisfies
$\tau \in \mathring{F}_{0}$.
\end{theorem}

Remark that $Z_{r,s}^{(2)}(\tau)\equiv \infty$ if $(r,s)=(0,0)$. To prove Theorems \ref{F-one}-\ref{Location}, we will "blow up" $Z_{r,s}^{(2)}(\tau)$ by considering $\lim_{s\to 0}\frac{1}{s}Z_{-Cs,s}^{(2)}(\tau)$, $C\in\mathbb{R}$, and the existence and uniqueness of $\tau(C)$ will follow from
that of the zero of $Z_{-Cs,s}^{(2)}(\tau)$ as $s\to 0$.

The rest of this paper is organized as follows. Theorem \ref{thm2} will be proved in Section \ref{modularform}. In Section \ref{tauc}, we apply Theorem \ref{thm2} to prove the existence and uniqueness of $\tau(C)$. See Theorem \ref{Unique-pole}.
In Section \ref{NOZ}, we give the detailed proofs of our main results Theorems \ref{eta-thm0}-\ref{Location} and Corollary \ref{line1/2}. Some precise characterizations of the three curves (see Theorem \ref{lemma-15}) will also be given. In Section \ref{geometricinter}, we introduce the relation between the three curves and the degeneracy curve of $G_2$ and prove the smoothness of the curves. Finally in Appendix A, we give another application of Theorem \ref{thm2}.

\section{Zeros of pre-modular forms}

\label{modularform}

This section is devoted to the proof of Theorem \ref{thm2}. First we recall the modularity of $g_2(\tau)$ and $\wp(z|\tau)$. Given any $\begin{pmatrix}
a & b\\
c & d
\end{pmatrix}
\in SL(2,\mathbb{Z})$, it is well known that
\begin{equation}g_{2}(\tfrac{a\tau+b}{c\tau+d})=(c\tau+d)^4g_{2}(\tau),\label{II-30-00}\end{equation}
\[
\wp \left(  \left.  \tfrac{z}{c\tau+d}\right \vert \tfrac{a\tau+b}{c\tau+d}\right)  =\left(
c\tau+d\right)  ^{2}\wp (z|\tau).\]
From here we can obtain
\[
\zeta \left(  \left.  \tfrac{z}{c\tau+d}\right \vert \tfrac{a\tau+b}{c\tau+d}\right)
=\left(  c\tau+d\right)  \zeta( z|\tau)  ,
\]
and so%
\begin{equation}%
\begin{pmatrix}
\eta_{2}(\tfrac{a\tau+b}{c\tau+d})\\
\eta_{1}(\tfrac{a\tau+b}{c\tau+d})
\end{pmatrix}
=(c\tau+d)\begin{pmatrix}
a & b\\
c & d
\end{pmatrix}
\begin{pmatrix}
\eta_{2}(\tau)\\
\eta_{1}(\tau)
\end{pmatrix}
. \label{II-31-1-00}%
\end{equation}
In the rest of this paper, we will freely use the formulas (\ref{II-30-00})-(\ref{II-31-1-00}).

As in \cite{Dahmen,LW2},  for any $(r,s)\in\mathbb{R}^{2}$, we define pre-modular forms
\begin{align}
Z_{r,s}(\tau)  :=&\zeta(r+s\tau|\tau)-r\eta_{1}(\tau)-s\eta_{2}%
(\tau)\nonumber \\
=&\zeta(r+s\tau|\tau)-(r+s\tau)\eta_{1}(\tau)+2\pi is, \label{c-7}%
\end{align}
\begin{equation}
Z_{r,s}^{(2)}(\tau):=Z_{r,s}(\tau)^{3}-3\wp(r+s\tau|\tau)Z_{r,s}(\tau
)-\wp^{\prime}(r+s\tau|\tau). \label{new3}%
\end{equation}
Since $\zeta(z|\tau)$ has simple poles at the lattice points
$\Lambda_{\tau}$, both $Z_{r,s}(\tau)$ and $Z_{r,s}^{(2)}(\tau)\equiv \infty$
provided $(r,s)  \equiv0$ mod $\mathbb{Z}^{2}$. If $(r,s)\in \frac{1}{2}\mathbb{Z}^{2}\backslash \mathbb{Z}^{2}$, where
\[\tfrac{1}{2}\mathbb{Z}^{2}:=\{(\tfrac{m}{2},\tfrac{n}{2}) \,|\, m,n\in\mathbb{Z}\},\]
then (\ref{quasi}) and the oddness of $\zeta(z|\tau)$ imply $Z_{r,s}%
(\tau)\equiv0$ and so $Z_{r,s}^{(2)}(\tau)\equiv0$, where we used $\wp'(\frac{\omega_k}{2})=0$. Therefore, we only consider $(
r,s)\in \mathbb{R}^{2}\backslash \frac{1}{2}\mathbb{Z}^{2}$. Then
both $Z_{r,s}(\tau)$ and $Z_{r,s}^{(2)}(\tau)$ are holomorphic in $\mathbb{H}$,
and it is easy to see that the following properties hold:

\begin{itemize}
\item[(i)] $Z_{r,s}(\tau)=\pm Z_{m\pm r,n\pm s}(\tau)$ and hence $Z_{r,s}^{(2)}(\tau)=\pm
Z_{m\pm r,n\pm s}^{(2)}(\tau)$ for any $(m,n)\in\mathbb{Z}^{2}$.

\item[(ii)] $Z_{r^{\prime},s^{\prime}}(\tau^{\prime})=(c\tau+d)Z_{r,s}(\tau)$
and hence $Z_{r^{\prime},s^{\prime}}^{(2)}(\tau^{\prime})=(c\tau+d)^{3}Z_{r,s}%
^{(2)}(\tau)$ for any $\gamma=
\begin{pmatrix}
a & b\\
c & d
\end{pmatrix}
\in SL(2,\mathbb{Z})$, where $\tau^{\prime}=\gamma \cdot \tau:=\frac{a\tau+b}{c\tau+d}$ and $(s^{\prime}%
,r^{\prime})=(s,r)\cdot \gamma^{-1}$.
\end{itemize}
\noindent In particular, when
$(r,s)\in Q_{N}$ is a $N$-torsion point for some $N\in\mathbb{N}_{\geq 3}$, where
\begin{equation}
Q_{N}:= \left \{  \left.  \left(  \tfrac{k_{1}}{N},\tfrac{k_{2}}%
{N}\right)  \right \vert \gcd(k_{1},k_{2},N)=1,\text{ }0\leq k_{1},k_{2}\leq
N-1\right \}  , \label{q-n}%
\end{equation}
and $\gamma\in \Gamma(N):=\{\gamma\in SL(2,\mathbb{Z})|\gamma\equiv I_2\operatorname{mod}N\}$, then $(r',s')\equiv(r,s)$ mod $\mathbb{Z}^2$. In other words, if $(r,s)\in Q_N$, then
\[Z_{r,s}\left(\tfrac{a\tau+b}{c\tau+d}\right)=(c\tau+d)Z_{r,s}(\tau),\quad Z_{r,s}^{(2)}\left(\tfrac{a\tau+b}{c\tau+d}\right)=(c\tau+d)^{3}Z_{r,s}%
^{(2)}(\tau)\]
hold for any $\gamma=\begin{pmatrix}
a & b\\
c & d
\end{pmatrix}\in\Gamma(N)$, namely $Z_{r,s}(\tau)$ and $Z_{r,s}^{(2)}(\tau)$ are \emph{modular forms} of weight $1$
and $3$, respectively, with respect to the principal congruence subgroup $\Gamma(N)$. Due to this reason, $Z_{r,s}(\tau)$ and $Z_{r,s}^{(2)}(\tau)$ are called \emph{pre-modular forms} in this paper as in \cite{LW2}.

We are interested in the zero structures of $Z_{r,s}(\tau)$ and $Z_{r,s}^{(2)}(\tau)$ for $(r,s)\in\mathbb{R}^2\setminus\frac{1}{2}\mathbb{Z}^2$.
By property (ii), we can restrict $\tau$ in the fundamental domain $F_{0}$ of $\Gamma
_{0}(2)$:
\[
F_{0}:=\{ \tau \in \mathbb{H}\ |\ 0\leqslant \  \text{Re}\  \tau \leqslant
1\  \text{and}\ |z-\tfrac{1}{2}|\geqslant \tfrac{1}{2}\},
\]
and by (i), we only need to consider $(r,s)\in \lbrack0,1]\times
\lbrack0,\frac{1}{2}]\backslash \frac{1}{2}\mathbb{Z}^{2}$. Recall the four open triangles defined in (\ref{rectangle}) (see Figure 1).
Clearly $[0,1]\times \lbrack0,\frac{1}{2}]=\cup_{k=0}^{3}\overline
{\triangle_{k}}$. The following result was proved in \cite{CKLW}.

\begin{figure}[btp]
\label{O5-1}\includegraphics[width=2.4in]{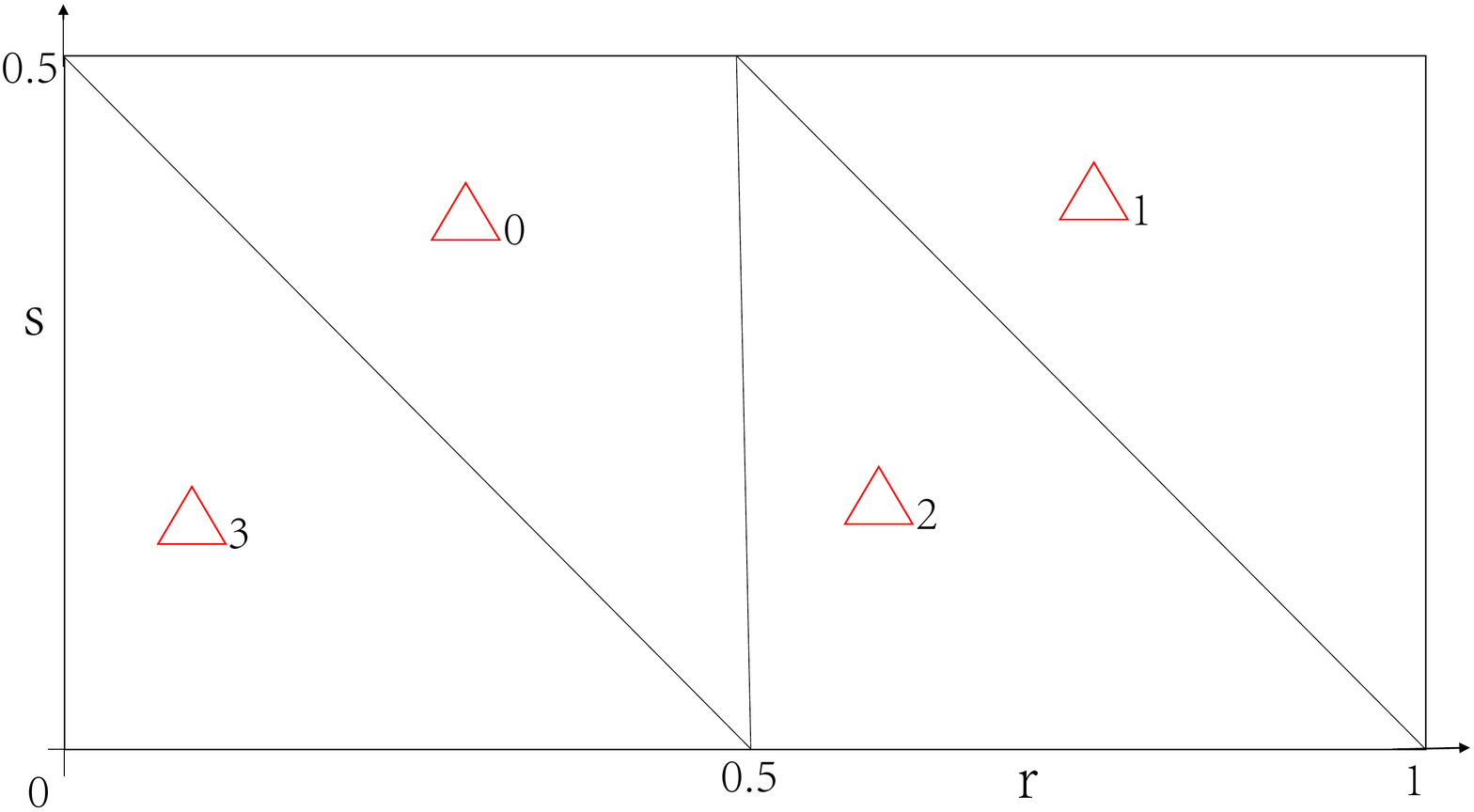}\caption{The four open trianlges $\triangle_{k}$.}%
\end{figure}

\medskip

\noindent{\bf Theorem A.} \cite{CKLW} {\it Let $(r,s)\in \lbrack0,1]\times \lbrack0,\frac{1}{2}]\backslash
\frac{1}{2}\mathbb{Z}^{2}$. Then $Z_{r,s}(\tau)=0$ has a solution $\tau$
in $F_{0}$ if and only if $(r,s)\in \triangle_{0}$. Furthermore, for any $(r,s)\in \triangle_{0}$, the zero $\tau \in F_{0}$ is unqiue and satisfies
$\tau \in \mathring{F}_{0}$}.
\medskip

In this paper, we will prove an analogous result for $Z_{r,s}^{(2)}(\tau)$.

\begin{theorem}[=Theorem \ref{thm2}]
\label{thm20}Let $(r,s)\in \lbrack0,1]\times \lbrack0,\frac{1}{2}]\backslash
\frac{1}{2}\mathbb{Z}^{2}$. Then $Z_{r,s}^{(2)}(\tau)=0$ has a solution $\tau$
in $F_{0}$ if and only if $(r,s)\in \triangle_{1}\cup \triangle_{2}\cup
\triangle_{3}$. Furthermore, for any $(r,s)\in \triangle_{1}\cup \triangle
_{2}\cup \triangle_{3}$, the zero $\tau \in F_{0}$ is unqiue and satisfies
$\tau \in \mathring{F}_{0}$.
\end{theorem}

Unlike $Z_{r,s}(\tau)$, Theorem \ref{thm20} shows an interesting phenomena for $Z_{r,s}^{(2)}(\tau)$. For example, $Z_{r,s}^{(2)}(\tau)$ has zeros in $F_0$ for $(r,s)\in\triangle_{1}\cup\triangle_{2}$, but it has no zeros in $F_0$ for $(r,s)\in \partial\triangle_{1}\cap\partial\triangle_{2}$.

The rest of this section is to prove Theorem \ref{thm20}. The reason why we choose the fundamental domain $F_0$ will be clear from the proof, particularly Lemma \ref{neq-zero}.
The basic strategy is similar to that of proving Theorem A in \cite{CKLW}. However, the argument is more involved and new techniques are needed. For example, for the same assertion of the pre-modular forms having no zero in $F_0$ for $(r,s)\in\cup_{k=0}^{3}\partial\triangle_{k}$, it is a trivial consequence of the same assertion for $(r,s)\in \triangle_{1}\cup \triangle
_{2}\cup \triangle_{3}$ in Theorem A; but obviously, this is not the case in Theorem \ref{thm20}.

First we need the following important results of PDE aspect.
\medskip

\noindent{\bf Theorem B.} \cite{LW2,CKL3}
\begin{itemize}
\item[(1)] \cite{LW2} {\it The mean field equation
\begin{equation}\label{mfe}\Delta u+e^u=16\pi\delta_0\quad\text{on}\; E_{\tau}:=\mathbb{C}/(\mathbb{Z}+\mathbb{Z}\tau)\end{equation}
has solutions if and only if there exists $(r,s)\in \mathbb{R}^{2}\setminus\frac{1}{2}\mathbb{Z}^{2}$ such that $\tau$ is a zero of $Z_{r,s}^{(2)}(\cdot)$}.

\item[(2)] \cite{CKL3} {\it If $\tau\in \{e^{\pi i/3}\}\cup i\mathbb{R}^+ $, then equation (\ref{mfe}) has no solutions}.
\end{itemize}

\begin{remark} In \cite{CLW,LW2}, Chai, Wang and the second author studied the following singular Liouville equation
\begin{equation}\label{mfe1}\Delta u+e^u=8n\pi\delta_0\quad\text{on}\; E_{\tau},\end{equation}
where $n\in\mathbb{N}$.
The solvability of (\ref{mfe1}) depends essentially on the moduli $\tau$ of the flat torus $E_{\tau}$ and is intricate from the PDE point of view. To settle this challenging problem, they studied it from the viewpoint of algebraic geometry. They
developed a theory to connect this PDE problem with the Lam\'{e} equation (\ref{lame}) and pre-modular forms. In particular, Wang and the second author \cite{LW2} proved the existence of a pre-modular form $Z_{r,s}^{(n)}(\cdot)$ of weight $\frac{n(n+1)}{2}$ such that (\ref{mfe1}) on $E_{\tau}$ has solutions if and only if $Z_{r,s}^{(n)}(\tau)=0$ for some $(r,s)\in \mathbb{R}^{2}\setminus\frac{1}{2}\mathbb{Z}^{2}$. Theorem B-(1) is a special case of this statement for $n=2$. Theorem B-(2) is a purely PDE result. We will see that Theorem B plays a crucial role in the proof of Lemma \ref{neq-zero} and hence Theorem \ref{thm20}. This is the only place related to the PDE result.
\end{remark}

\begin{lemma}
\label{neq-zero}Let $(r,s)\in \lbrack0,1]\times \lbrack0,\frac{1}{2}%
]\backslash \frac{1}{2}\mathbb{Z}^{2}$. Then $Z_{r,s}^{(2)}(\tau)\not =0$ for
any $\tau \in \{e^{\pi i/3}\} \cup(\partial F_{0}\cap \mathbb{H})$.
\end{lemma}

\begin{proof}
It does not seem that this assertion could be obtained directly from the expression (\ref{new3})
of $Z_{r,s}^{(2)}(\tau)$. Indeed, this lemma is a consequence of the result of \emph{nonlinear
PDEs} (i.e. Theorem B).

Given $\tau \in \{e^{\pi i/3}\} \cup(\partial F_{0}\cap \mathbb{H})$. If $\tau \in
\{ e^{\pi i/3} \} \cup i\mathbb{R}^{+}$, then Theorem B
implies $Z_{r,s}^{(2)}(\tau)\not =0$ for any $(r,s)\in \mathbb{R}^{2}%
\backslash \frac{1}{2}\mathbb{Z}^{2}$.
If $\tau \in i\mathbb{R}^{+}+1$, then
by applying $\gamma=%
\begin{pmatrix}
1 & -1\\
0 & 1
\end{pmatrix}
$ in property (ii), we have that $\tau-1\in i\mathbb{R}^{+}$ and
\[
Z_{r,s}^{(2)}(\tau)=Z_{r+s,s}^{(2)}(\tau-1)\not =0\text{ for any }%
(r,s)\in \mathbb{R}^{2}\backslash \tfrac{1}{2}\mathbb{Z}^{2}.
\]
If $|\tau-\frac{1}{2}|=\frac{1}{2}$, then again by applying $\gamma=%
\begin{pmatrix}
1 & 0\\
-1 & 1
\end{pmatrix}
$ in property (ii)  we see that $\frac{\tau}{1-\tau}\in i\mathbb{R}^{+}$ and
\[
(1-\tau)^{3}Z_{r,s}^{(2)}(\tau)=Z_{r,r+s}^{(2)}(\tfrac{\tau}{1-\tau}%
)\not =0\text{ for any }(r,s)\in \mathbb{R}^{2}{\small \backslash}\tfrac{1}%
{2}\mathbb{Z}^{2}.
\]
This completes the proof.
\end{proof}

Recalling $Q_N$ in (\ref{q-n}), we define%
\begin{equation}
M_{N}(\tau):=\prod_{(r,s)\in Q_{N}}Z_{r,s}^{(2)}(\tau).\label{f-N}%
\end{equation}
By properties (i)-(ii), it is easy to see that $M_{N}(\tau)$ is a \emph{modular form} with respect to
$SL(2,\mathbb{Z})$ of weight $3|Q_{N}|$ (i.e. for any $\gamma\in SL(2,\mathbb{Z})$, when $(r,s)$ runs over all elements of $Q_{N}$, then so does $(r',s')$ after modulo $\mathbb{Z}^2$), where $|Q_{N}|=\#Q_{N}$.
To apply the theory of modular forms, we recall the following classical formula. See e.g. \cite{Diamond-Shurman,Serre} for the proof.\medskip

\noindent \textbf{Theorem C.} \textit{Let $f(\tau)$ be a nonzero
modular form with respect to $SL(2,\mathbb{Z})$\ of weight
$k$. Then}
\begin{equation}
\sum_{\tau \in \mathbb{H}\backslash \{i,\rho \}}\nu_{\tau}(f)+\nu_{\infty
}(f)+\frac{1}{2}\nu_{i}(f)+\frac{1}{3}\nu_{\rho}(f)=\frac{k}{12},
\label{modular}%
\end{equation}
\textit{where $\rho:=e^{\pi i/3}$, $\nu_{\tau}(f)$ denotes
the zero order of $f$ at $\tau$ and the summation over
$\tau$ is performed modulo $SL(2,\mathbb{Z})$
equivalence.}\medskip

We note for each $(r,s)\in Q_{N}$, there exists a unique $(\tilde{r},\tilde{s})\in Q_N$ such that $(\tilde{r},\tilde{s})\equiv(-r,-s)$ mod $\mathbb{Z}^{2}$. Then property (i) of
$Z_{r,s}^{(2)}(\tau)$ gives $Z_{\tilde{r},\tilde{s}}^{(2)}(\tau
)=-Z_{r,s}^{(2)}(\tau)$, which implies that%
\begin{equation}
\nu_{\tau}(M_{N})\in2\mathbb{N}\cup \{0\} \text{ \ for any \ }\tau \in
\mathbb{H}.\label{mo-2}%
\end{equation}

To apply Theorem C, we need the asymptotics of
$Z_{r,s}^{(2)}(\tau)$ as $\operatorname{Im}\tau \rightarrow+\infty$.

\begin{lemma}
\label{infinity-behavior copy(1)}Let $(r,s)\in \lbrack0,1)\times \lbrack
0,1)\backslash \frac{1}{2}\mathbb{Z}^{2}$ and $q=e^{2\pi i\tau}$ with $\tau\in F_0$. Then the asymptotics of $Z_{r,s}^{(2)}(\tau)$ at the three
cusps $\tau$ $=0,1,\infty$ are as follows:
\end{lemma}

\begin{itemize}
\item[(a)] As $F_0\ni \tau \rightarrow\infty$,
\[
Z_{r,s}^{(2)}(\tau)=4\pi^{3}is(1-s)(2s-1)+o(1)\;\text{ if }\;s\in \left(
0,\tfrac{1}{2}\right)  \cup \left(  \tfrac{1}{2},1\right)  ,
\]%
\[
Z_{r,s}^{(2)}(\tau)=-48\pi^{3}\sin(2\pi r)q+O(q^{2})\;\text{ if }\; s=0,
\]%
\[
Z_{r,s}^{(2)}(\tau)=-12\pi^{3}\sin(2\pi r)q^{1/2}+O(q)\;\text{ \ if \ }\; s=1/2.
\]

\item[(b)] As $F_0\ni \tau \rightarrow0$,
\[
\lim_{\tau\to 0}Z_{r,s}^{(2)}(\tau)= \infty\;\text{ if }\;%
r\in \left(  0,\tfrac{1}{2}\right)  \cup \left(  \tfrac{1}{2},1\right)  .
\]

\item[(c)] As $F_0\ni \tau \rightarrow1$,
\[
\lim_{\tau\to 1}Z_{r,s}^{(2)}(\tau) =\infty\;\text{ if }\;\left(
r+s\right)  \in \left(  0,\tfrac{1}{2}\right)  \cup \left(  \tfrac{1}{2},1\right)
\cup \left(  1,\tfrac{3}{2}\right)  .
\]

\end{itemize}

\begin{proof}
By using the $q$-expansions of $\wp(z|\tau)$ and $Z_{r,s}(\tau)$
(see (\ref{q-e1})-(\ref{zz}) in Section \ref{tauc}), the asymptotics of $Z_{r,s}^{(2)}(\tau)$ as
$\tau\to\infty$ can be easily calculated. Because the calculation is straightforward
and is already done in \cite{Dahmen,LW2}, we omit the details for (a) here.

The asymptotics of $Z_{r,s}^{(2)}(\tau)$ at the cusp $0$ can be obtained
by using property (ii) and the assertion (a). Letting $\gamma=%
\begin{pmatrix}
1 & -1\\
1 & 0
\end{pmatrix}
$ leads to
\[
Z_{r+s,-r}^{(2)}(\tfrac{\tau-1}{\tau})=\tau^{3}Z_{r,s}^{(2)}(\tau).
\]
When $\tau \in F_{0}$ and $\tau \rightarrow0$, we have $\tfrac{\tau-1}{\tau}\in
F_{0}$ and $\tfrac{\tau-1}{\tau}\rightarrow \infty$. Then applying (a) we obtain that as $F_0\ni\tau \rightarrow0$,%
\begin{align}
Z_{r,s}^{(2)}(\tau) &  =\frac{-1}{\tau^{3}}Z_{-(r+s),r}^{(2)}(\tfrac{\tau
-1}{\tau})\nonumber \\
&  =\frac{-1}{\tau^{3}}\left[  4\pi^{3}ir(1-r)(2r-1)+o(1)\right]  \text{ \ if
\ }r\in(0,\tfrac{1}{2})\cup(\tfrac{1}{2},1).\label{ayp-8}%
\end{align}
This proves (b).

Similarly, when $\tau\in F_0$ and $\tau \rightarrow1$, we have $\tfrac{\tau-1}{\tau}\in F_{0}$
and $\tfrac{\tau-1}{\tau}\rightarrow0$. Applying property (ii) and (\ref{ayp-8})
we obtain that as $F_0\ni\tau \rightarrow1$,%
\begin{align*}
Z_{r,s}^{(2)}(\tau) &  =\frac{1}{\tau^{3}}Z_{r+s,-r}^{(2)}(\tfrac{\tau-1}%
{\tau})\\
&  =\frac{-1}{(\tau-1)^{3}}\left[  4\pi^{3}i(r+s)(1-r-s)(2r+2s-1)+o(1)\right]
\end{align*}
for $r+s\in(0,\tfrac{1}{2})\cup(\tfrac{1}{2},1)$. The remaining case $r+s\in(1,\tfrac{3}{2})$ follows from
$Z_{r,s}^{(2)}(\tau)  =Z_{r-1,s}^{(2)}(\tau)$. This proves (c).
\end{proof}

Lemma \ref{infinity-behavior copy(1)}-(a) implies
\begin{equation}
\text{the vanishing order of }Z_{r,s}^{(2)}(\tau)\text{ at }\infty \text{ is
}\left \{
\begin{array}
[c]{l}%
0\text{ \ if \ }s\not =0,1/2\\
1\text{ \ if \ }s=0\\
\frac{1}{2}\text{ \ if \ }s=1/2
\end{array}
\right.  .\label{vanish order}%
\end{equation}
Recall $\triangle_{k}$, $k=0,1,2,3$, defined in (\ref{rectangle}).

\begin{lemma}
\label{yl-1}Fix $k\in \{0,1,2,3\}$. Then the number of zeros of $Z_{r,s}%
^{(2)}(\tau)$ in $F_{0}$ is a constant for $(r,s)\in \triangle_{k}$.
\end{lemma}

\begin{proof}
Since $(r,s)\in \triangle_{k}$, we have $r,s,r+s\not \in \{0,\frac{1}%
{2},1,\frac{3}{2}\}$, so Lemma \ref{infinity-behavior copy(1)} (a)-(c) imply
that%
\[
Z_{r,s}^{(2)}(\tau)\not \rightarrow 0\;\text{ as }\;F_{0}\ni \tau \rightarrow
\infty,0,1
\]
respectively. Together with Lemma \ref{neq-zero} that $Z_{r,s}^{(2)}%
(\tau)\not =0$ on $\partial F_{0}\cap \mathbb{H}$, it is easy to apply the
argument principle to conclude that the number of zeros of $Z_{r,s}^{(2)}%
(\tau)$ in $F_{0}$ is a constant for $(r,s)\in \triangle_{k}$.
\end{proof}

\begin{lemma}
\label{lem-1}Let $(r,s)\in Q_{3}$. Then $Z_{r,s}^{(2)}(\tau)\not =0$ for any
$\tau \in \mathbb{H}$.
\end{lemma}

\begin{proof}
Note that $3|Q_{3}|=24$. Since $s\not =\frac{1}{2}$ and\ $(\frac
{1}{3},0)$, $(\frac{2}{3},0)\in Q_{3}$, we see from Lemma
\ref{infinity-behavior copy(1)}-(a) (or (\ref{vanish order})) that $M_{3}(\tau)\sim q^{2}$ as $F_{0}\ni
\tau \rightarrow \infty$, i.e. $\nu_{\infty}(M_{3})=2$. Therefore, we deduce
from (\ref{modular}) that $M_{3}(\tau)$ has no zeros in $\mathbb{H}$.
\end{proof}

\begin{figure}[btp]\label{fundamental-dom}
\includegraphics[width=3.6in]{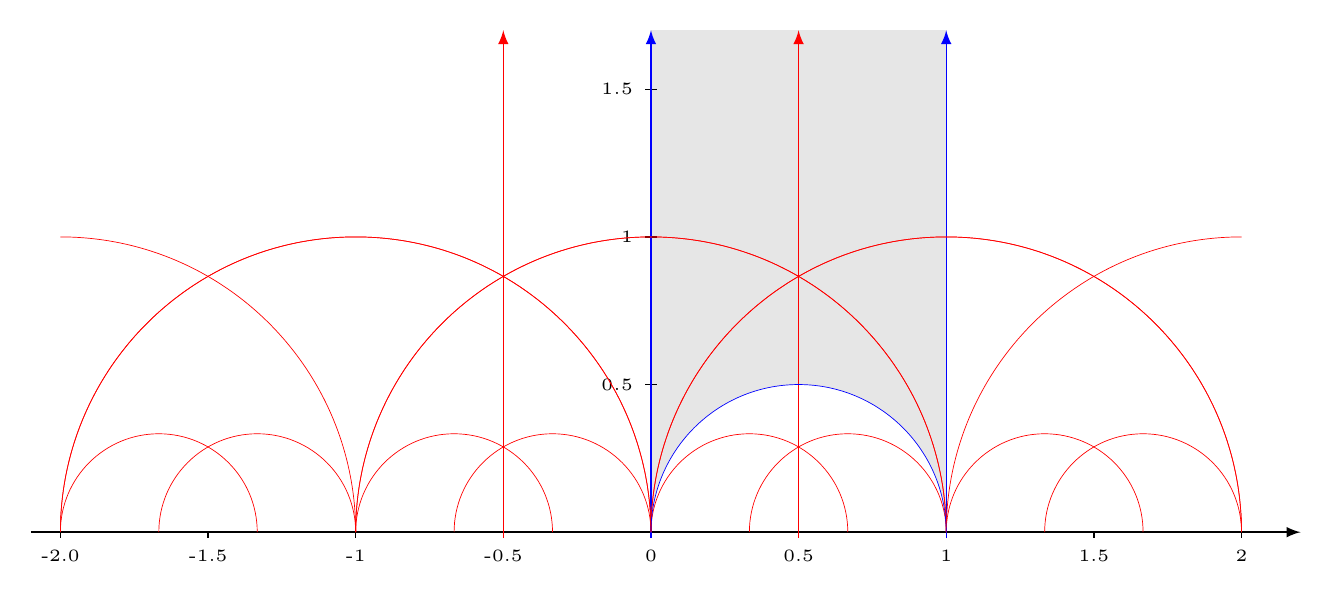}\caption{$F_{0}=F\cup
\gamma_{1}(F)\cup\gamma_{2}(F)$.}%
\end{figure}

Let $F$ be a fundamental domain of $SL(2,\mathbb{Z})$ defined by\footnote{Of course, the standard definition of $F$ should be $F:=\{ \tau \in \mathbb{H}\ |\ 0\leq \operatorname{Re}\tau< 1,|\tau|\geq
1,|\tau-1|>1\}\cup\{e^{\pi i/3}\}$, i.e. $\operatorname{Re}\tau=1$ is not needed. In this paper, to guarantee the validity of (\ref{F-de}), it is more convenient for us to use the definition (\ref{F-fun}), which does not effect our following argument because Lemma \ref{neq-zero} says that $Z_{r,s}^{(2)}(\tau)\neq 0$ if $\tau\in \{\tau\in \mathbb{H} |\operatorname{Re}\tau=1\}\cup \gamma_{1}(\{\tau\in \mathbb{H} |\operatorname{Re}\tau=1\})\cup\gamma_{2}(\{\tau\in \mathbb{H} |\operatorname{Re}\tau=1\})\subset \partial F_{0}\cap \mathbb{H}$.}
\begin{equation}\label{F-fun}
F:=\{ \tau \in \mathbb{H}\ |\ 0\leq \operatorname{Re}\tau\le 1,|\tau|\geq
1,|\tau-1|>1\}\cup\{e^{\pi i/3}\}.
\end{equation}
Define $\gamma(F):=\{
\gamma \cdot \tau|\tau \in F\}$ for any $\gamma \in SL(2,\mathbb{Z})$, then
$\gamma(F)$ is also a fundamental domain of $SL(2,\mathbb{Z})$. Let%
\begin{equation}
\gamma_{1}:=%
\begin{pmatrix}
0 & 1\\
-1 & 1
\end{pmatrix}
,\  \  \gamma_{2}:=%
\begin{pmatrix}
1 & -1\\
1 & 0
\end{pmatrix}
, \label{gam}%
\end{equation}
then it is easy to prove that
\begin{equation}
F_{0}=F\cup \gamma_{1}(F)\cup \gamma_{2}(F). \label{F-de}%
\end{equation}
See Figure 2, which is copied from \cite{CKLW}.
Now we are in the position to prove Theorem \ref{thm20}.

\begin{proof}
[Proof of Theorem \ref{thm20}]Recall Lemma \ref{neq-zero} that
$Z_{r,s}^{(2)}(\tau)\neq0$ for $\tau \in \{ \rho \}\cup \left(
\partial F_{0}\cap \mathbb{H}\right)  $. We divide the proof of Theorem
\ref{thm20} into several steps.\medskip

\textbf{Step 1.} We claim that $Z_{r,s}^{(2)}(\tau)$ has no zeros in $
F_{0}$ for $(r,s)\in \triangle_{0}$.\medskip

Lemma \ref{lem-1} says that $Z_{\frac{1}{3},\frac{1}{3}}^{(2)}(\tau)$ has no
zeros in $F_{0}$. Since $(\frac{1}{3},\frac{1}{3})\in \triangle_{0}$, our
claim follows directly from Lemma \ref{yl-1}.\medskip

\textbf{Step 2.} We claim that $Z_{r,s}^{(2)}(\tau)$ has a unique zero in
$F_{0}$ for $(r,s)\in \triangle_{1}\cup \triangle_{2}\cup \triangle_{3}$.\medskip

Since
\begin{equation}
(\tfrac{5}{6},\tfrac{1}{3})\in \triangle_{1},\text{ \ }(\tfrac{2}{3},\tfrac
{1}{6})\in \triangle_{2},\text{ \ }(\tfrac{1}{6},\tfrac{1}{6})\in \triangle
_{3},\label{fc-16}%
\end{equation}
by Lemma \ref{yl-1} we only need to prove the claim for $(r,s)$ $\in
\{(\tfrac{1}{6},\tfrac{1}{6}),$ $(\tfrac{2}{3},\tfrac{1}{6}),$ $(\tfrac{5}%
{6},\tfrac{1}{3})\}$.
Note that $M_{6}(\tau)$ is a modular form of weight
$3|Q_{6}|=72$.
For $(  r,s)\in Q_{6}$, Lemma
\ref{infinity-behavior copy(1)}-(a) shows that $\underset{\tau \rightarrow \infty}{\lim
}Z_{r,s}^{(2)}(\tau)$ $=0$ if and only if $\left(  r,s\right)\in \{(\tfrac{1}{6},0),(\tfrac{5}{6},0),(\tfrac{1}{6},\tfrac{1}{2}),(\tfrac{2}%
{6},\tfrac{1}{2}),(\tfrac{4}{6},\tfrac{1}{2}),(\tfrac{5}{6},\tfrac{1}{2})\}$. Thus we see from (\ref{vanish order}) that
\[
\nu_{\infty}(M_{6})=1\times2+\tfrac{1}{2}\times4=4.
\]
On the other hand, Lemma \ref{neq-zero} says $\nu_{i}(M_{6})=\nu_{\rho}%
(M_{6})=0$. Therefore, it follows from (\ref{modular}) that%
\begin{equation}
\sum_{\tau \in \mathbb{H}\backslash \{i,\rho \}}\nu_{\tau}(M_{6})=2.\label{eq-1}%
\end{equation}
Recall that $F$ defined in (\ref{F-fun}) is a
fundamental domain of $SL(2,\mathbb{Z})$. Applying (\ref{mo-2}), there exists a \emph{unique} $\tau_{0}\in
(F\cap \mathring{F}_{0})\setminus\{\rho=e^{\pi i/3}\}$ such that $\nu_{\tau_{0}}(M_{6})=2$, i.e. there exists
a \emph{unique} $(r_{1},s_{1})\in Q_{6}^{\prime}$ such that
\begin{equation}
Z_{r_{1},s_{1}}^{(2)}(\tau_{0})=0,\label{h}%
\end{equation}
where
\[
Q_{6}^{\prime}:=\left \{
\begin{array}
[c]{c}%
(0,\tfrac{1}{6}),(\tfrac{1}{6},0),(\tfrac{1}{6},\tfrac{1}{6}),(\tfrac{1}%
{6},\tfrac{1}{3}),(\tfrac{1}{6},\tfrac{1}{2}),(\tfrac{1}{3},\tfrac{1}{6}),\\
(\tfrac{1}{3},\tfrac{1}{2}),(\tfrac{1}{2},\tfrac{1}{6}),(\tfrac{1}{2}%
,\tfrac{1}{3}),(\tfrac{2}{3},\tfrac{1}{6}),(\tfrac{5}{6},\tfrac{1}{6}%
),(\tfrac{5}{6},\tfrac{1}{3})
\end{array}
\right \}\subset [0,1)\times[0,\tfrac{1}{2}].
\]
Remark that for any $(r,s)\in Q_6$, either $(r,s)\in Q_6'$ or there exists a unique $(\tilde{r},\tilde{s})\in Q_6'$ such that $(r,s)\equiv(-\tilde{r},-\tilde{s})$ mod $\mathbb{Z}^{2}$.

\textbf{Step 2-1.} We prove that $(r_{1},s_{1})\in \{(\tfrac{1}{6},\tfrac{1}%
{6}),(\tfrac{2}{3},\tfrac{1}{6}),(\tfrac{5}{6},\tfrac{1}{3})\}$.\medskip

Assume by contradiction that $(r_{1},s_{1})\not \in \{(\tfrac{1}%
{6},\tfrac{1}{6}),(\tfrac{2}{3},\tfrac{1}{6}),(\tfrac{5}{6},\tfrac{1}{3})\}$.
Then by (\ref{eq-1})-(\ref{h}), we have
\begin{equation}
Z_{r,s}^{(2)}(\tau)\not =0\text{ in }F\text{ for }(r,s)\in \{(\tfrac{1}%
{6},\tfrac{1}{6}),(\tfrac{2}{3},\tfrac{1}{6}),(\tfrac{5}{6},\tfrac{1}%
{3})\}.\label{fc-9}%
\end{equation}
Recall (\ref{gam})-(\ref{F-de}). Letting $\gamma=\gamma_{1}=%
\begin{pmatrix}
0 & 1\\
-1 & 1
\end{pmatrix}
$ in property (ii) leads to%
\begin{equation}
Z_{-s,r+s}^{(2)}(\gamma_{1}\cdot \tau)=(1-\tau)^{3}Z_{r,s}^{(2)}(\tau
).\label{fc-31}%
\end{equation}
Applying this to $(r,s)$ $\in \{(\tfrac{1}{6},\tfrac{1}{6}),(\tfrac{2}%
{3},\tfrac{1}{6}),(\tfrac{5}{6},\tfrac{1}{3})\}$, it follows from property (i)
that%
\begin{equation}
Z_{\frac{5}{6},\frac{1}{3}}^{(2)}(\gamma_{1}\cdot \tau)=Z_{\frac{-1}{6}%
,\frac{1}{3}}^{(2)}(\gamma_{1}\cdot \tau)=(1-\tau)^{3}Z_{\frac{1}{6},\frac
{1}{6}}^{(2)}(\tau),\label{fc-12}%
\end{equation}%
\begin{equation}
-Z_{\frac{1}{6},\frac{1}{6}}^{(2)}(\gamma_{1}\cdot \tau)=Z_{\frac{-1}{6}%
,\frac{5}{6}}^{(2)}(\gamma_{1}\cdot \tau)=(1-\tau)^{3}Z_{\frac{2}{3},\frac
{1}{6}}^{(2)}(\tau),\label{fc121}%
\end{equation}%
\begin{equation}
Z_{\frac{2}{3},\frac{1}{6}}^{(2)}(\gamma_{1}\cdot \tau)=Z_{\frac{-1}{3}%
,\frac{7}{6}}^{(2)}(\gamma_{1}\cdot \tau)=(1-\tau)^{3}Z_{\frac{5}{6},\frac
{1}{3}}^{(2)}(\tau).\label{fc-13}%
\end{equation}
Together with (\ref{fc-9}), we obtain
\begin{equation}
Z_{r,s}^{(2)}(\tau)\not =0\text{ in }\gamma_{1}(F)\text{ for }(r,s)\in
\{(\tfrac{1}{6},\tfrac{1}{6}),(\tfrac{2}{3},\tfrac{1}{6}),(\tfrac{5}{6}%
,\tfrac{1}{3})\}.\label{fc-10}%
\end{equation}
Similarly, letting $\gamma=\gamma_{2}=%
\begin{pmatrix}
1 & -1\\
1 & 0
\end{pmatrix}
$ in property (ii) leads to%
\begin{equation}
Z_{r+s,-r}^{(2)}(\gamma_{2}\cdot \tau)=\tau^{3}Z_{r,s}^{(2)}(\tau)
\label{fc-32}%
\end{equation}
and so%
\begin{equation}
-Z_{\frac{2}{3},\frac{1}{6}}^{(2)}(\gamma_{2}\cdot \tau)=Z_{\frac{1}{3}%
,\frac{-1}{6}}^{(2)}(\gamma_{2}\cdot \tau)=\tau^{3}Z_{\frac{1}{6},\frac{1}{6}%
}^{(2)}(\tau), \label{fc-14}%
\end{equation}%
\begin{equation}
Z_{\frac{5}{6},\frac{1}{3}}^{(2)}(\gamma_{2}\cdot \tau)=Z_{\frac{5}{6}%
,\frac{-2}{3}}^{(2)}(\gamma_{2}\cdot \tau)=\tau^{3}Z_{\frac{2}{3},\frac{1}{6}%
}^{(2)}(\tau), \label{fc141}%
\end{equation}%
\begin{equation}
Z_{\frac{1}{6},\frac{1}{6}}^{(2)}(\gamma_{2}\cdot \tau)=Z_{\frac{7}{6}%
,\frac{-5}{6}}^{(2)}(\gamma_{2}\cdot \tau)=\tau^{3}Z_{\frac{5}{6},\frac{1}{3}%
}^{(2)}(\tau). \label{fc-15}%
\end{equation}
Together with (\ref{fc-9}), we obtain
\[
Z_{r,s}^{(2)}(\tau)\not =0\text{ in }\gamma_{2}(F)\text{ for }(r,s)\in
\{(\tfrac{1}{6},\tfrac{1}{6}),(\tfrac{2}{3},\tfrac{1}{6}),(\tfrac{5}{6}%
,\tfrac{1}{3})\}.
\]
Therefore, it follows from (\ref{F-de}) (i.e.
$F_{0}=F\cup \gamma_{1}(F)\cup \gamma_{2}(F)$)
that $Z_{r,s}^{(2)}(\tau)\not =0$ in $F_{0}$ for $(r,s)$
$\in \{(\tfrac{1}{6},\tfrac{1}{6}),(\tfrac{2}{3},\tfrac{1}{6}),(\tfrac{5}%
{6},\tfrac{1}{3})\}.$
By (\ref{fc-16}), we conclude from Lemma \ref{yl-1} and Step 1 that%
\begin{equation}
Z_{r,s}^{(2)}(\tau)\not =0\text{ \ in \ }F_{0}\text{ \ for any \ }(r,s)\in \cup
_{k=0}^{3}\triangle_{k}. \label{fc-11}%
\end{equation}
From (\ref{fc-11}), (\ref{h}) and  $(r_{1},s_{1})\in Q'_{6}\subset [0,1)\times
[0,\frac{1}{2}]$, we obtain $(r_{1},s_{1})\in \cup_{k=0}^{3}\partial
\triangle_{k}$. This, together with $Z_{r_{1},s_{1}}^{(2)}(\tau_{0})=0$ and
the argument principle, implies the existence of $(r,s)\in \cup_{k=0}%
^{3}\triangle_{k}$ close to $(r_{1},s_{1})$ such that $Z_{r,s}^{(2)}(\tau)=0$
for some $\tau \in F_{0}$ close to $\tau_{0}$, a contradiction with
(\ref{fc-11}). Therefore, $(r_{1},s_{1})$ $\in \{(\tfrac{1}{6},\tfrac{1}%
{6}),(\tfrac{2}{3},\tfrac{1}{6}),(\tfrac{5}{6},\tfrac{1}{3})\}.$

\textbf{Step 2-2.} We prove that $Z_{r,s}^{(2)}(\tau)$ has a unique zero in
$F_{0}$ for each $(r,s)\in \{(\tfrac{1}{6},\tfrac{1}{6}),(\tfrac{2}{3},\tfrac
{1}{6}),(\tfrac{5}{6},\tfrac{1}{3})\}$.\medskip

By Step 2-1, without loss of generality, we may assume $(r_{1},s_{1}%
)=(\tfrac{1}{6},\tfrac{1}{6})$ (the other two cases $(r_{1},s_{1})=(\tfrac
{2}{3},\tfrac{1}{6}),(\tfrac{5}{6},\tfrac{1}{3})$ can be discussed in the same way).

Then $\tau_{0}$ is the unique zero of $Z_{\frac{1}{6},\frac{1}{6}}%
^{(2)}(\tau)$ in $F$ and (\ref{eq-1}) implies
\begin{equation}
Z_{r,s}^{(2)}(\tau)\not =0\text{ in }F\text{ for }(r,s)\in \{(\tfrac{2}%
{3},\tfrac{1}{6}),(\tfrac{5}{6},\tfrac{1}{3})\}. \label{s}%
\end{equation}
This together with (\ref{fc121}) and (\ref{fc-15}) implies
\[
Z_{\frac{1}{6},\frac{1}{6}}^{(2)}(\tau)\not =0\text{ in }\gamma_{1}%
(F)\cup \gamma_{2}(F).
\]
Therefore, by (\ref{F-de}) we conclude that $Z_{\frac{1}{6},\frac{1}{6}}^{(2)}(\tau)$ has
a unique zero in $F_{0}$.

For $\left(  r,s\right)  =\left(  \frac{5}{6},\frac{1}{3}\right)  $, by
applying (\ref{fc-12}), we see that $\gamma_{1}\cdot \tau_{0}$ is the unique
zero of $Z_{\frac{5}{6},\frac{1}{3}}^{(2)}(\tau)$ in $\gamma_{1}(F)$. Clearly
(\ref{fc141}) and (\ref{s}) give
\[
Z_{\frac{5}{6},\frac{1}{3}}^{(2)}(\tau)\not =0\text{ in }\gamma_{2}(F).
\]
Together with (\ref{s}) and (\ref{F-de}), we conclude that $Z_{\frac{5}{6},\frac{1}{3}}^{(2)}(\tau)$ has a unique
zero in $F_{0}$. By a similar discussion, $Z_{\frac{2}{3},\frac{1}{6}}%
^{(2)}(\tau)$ also has a unique zero in $F_{0}$.

Now by (\ref{fc-16}) and Lemma \ref{yl-1}, we conclude that $Z_{r,s}%
^{(2)}(\tau)$ has a unique zero in $F_{0}$ for $(r,s)\in \triangle_{1}%
\cup \triangle_{2}\cup \triangle_{3}$. This proves Step 2.\medskip

\textbf{Step 3.} We prove that $Z_{r,s}^{(2)}(\tau)\not =0$ in $F_{0}$ for any
$(r,s)\in \cup_{k=0}^{3}\partial \triangle_{k}\backslash \frac{1}{2}%
\mathbb{Z}^{2}$.\medskip

Suppose that there exists $(r_{0},s_{0})\in \cup_{k=0}^{3}\partial \triangle
_{k}\backslash \frac{1}{2}\mathbb{Z}^{2}$ such that $Z_{r_{0},s_{0}}^{(2)}%
(\tau)=0$ has a zero $\tau_{0}$ in $F_{0}$. Lemma \ref{neq-zero} implies
$\tau_{0}\in \mathring{F}_{0}\setminus\{\rho\}$. Clearly there exists a sequence of \emph{prime
numbers} $N\rightarrow \infty$ and $(\tilde{r}_{N},\tilde{s}_{N})\in Q_{2N},$
$\tilde{s}_{N}\leq \frac{1}{2}$, such that%
\[
(\tilde{r}_{N},\tilde{s}_{N})\in \cup_{k=0}^{3}\partial \triangle_{k}%
\backslash \tfrac{1}{2}\mathbb{Z}^{2}\text{ and }(\tilde{r}_{N},\tilde{s}%
_{N})\rightarrow(r_{0},s_{0})\text{ as }N\rightarrow \infty.
\]
Again by the argument principle, it follows that%
\begin{equation}
Z_{\tilde{r}_{N},\tilde{s}_{N}}^{(2)}(\tau)\text{ has a zero }\tilde{\tau}%
_{N}\in \mathring{F}_{0}\setminus\{\rho\}\text{ for }N\text{ large.}\label{kl}%
\end{equation}
By (\ref{F-de}), (\ref{fc-31}) and (\ref{fc-32}), we may always assume $\tilde{\tau
}_{N}\in F$ by replacing $\tilde{\tau}_{N}$ with one of $\{ \gamma_{1}%
^{-1}\cdot \tilde{\tau}_{N},$ $\gamma_{2}^{-1}\cdot \tilde{\tau}_{N}\}$ if necessary.

Now fix such a large prime number $N$. Recalling (\ref{f-N}) that $M_{2N}%
(\tau)$ is a modular form with respect to $SL(2,\mathbb{Z})$ of weight
$3|Q_{2N}|=9(N^{2}-1)$. Since for any $k\in \{1,3,\cdot \cdot \cdot,2N-1\}
\backslash \{N\}$, $(\frac{k}{2N},0)\in Q_{2N}$, and for any $k\in
\{1,2,\cdot \cdot \cdot,2N-1\} \backslash \{N\}$, $(\frac{k}{2N},\frac{1}{2})\in
Q_{2N}$, it follows from Lemma \ref{infinity-behavior copy(1)}-(a) that
\[
\nu_{\infty}(M_{2N})=1\times(N-1)  +\tfrac{1}{2}\times
2(N-1)=2(N-1).
\]
Together with Lemma \ref{neq-zero} that $\nu_{i}(M_{2N})=\nu_{\rho}(M_{2N}%
)=0$, we see from (\ref{modular}) that%
\begin{equation}
\sum_{\tau \in \mathbb{H}\backslash \{i,\rho \}}\nu_{\tau}(M_{N})=\frac
{3(N^{2}-1)}{4}-2(N-1)=\frac{3N^{2}-8N+5}{4}, \label{fc-23}%
\end{equation}
where the summation over $\tau$ is performed modulo $SL(2,\mathbb{Z})$ equivalence.

On the other hand, recall%
\[
\triangle_{3}=\{(r,s)\mid r>0,\text{ }s>0,\text{ }r+s<\tfrac{1}{2}\}.
\]
It is easy to compute that in $Q_{2N}$, there are%
\[
1+2+\cdot \cdot \cdot+\frac{N-3}{2}=\frac{(N-1)(N-3)}{8}%
\]
$(r,s)=(\frac{k_{1}}{2N},\frac{k_{2}}{2N})$'s belonging to $\triangle_{3}$
such that $k_{1}$ is odd and $k_{2}$ is even (resp. $k_{1}$ is even and
$k_{2}$ is odd); and there are%
\[
1+2+\cdot \cdot \cdot+\frac{N-1}{2}=\frac{(N+1)(N-1)}{8}%
\]
$(r,s)=(\frac{k_{1}}{2N},\frac{k_{2}}{2N})$'s belonging to $\triangle_{3}$
such that $k_{1}$ and $k_{2}$ are both odd. Thus%
\begin{align}
|Q_{2N}\cap \triangle_{3}| &  =\frac{(N-1)(N-3)}{4}+\frac{(N+1)(N-1)}%
{8}\label{mo-3}\\
&  =\frac{3N^{2}-8N+5}{8}.\nonumber
\end{align}
Write%
\[
Q_{2N}\cap \triangle_{3}=\left \{  (r_{k},s_{k})\left \vert 1\leq k\leq
\frac{3N^{2}-8N+5}{8}\right.  \right \}  .
\]
Applying Step 2, we see that $Z_{r_{k},s_{k}}(\tau)$ has a unique zero
$\tau_{k}\in \mathring{F}_{0}$ in $F_{0}$. If $\tau_{k}\in \gamma_{1}%
(F)\cup \gamma_{2}(F)$, say $\tau_{k}\in \gamma_{1}(F)$ for example, for some
$k$, then by (\ref{fc-31}) and property (i) it is easy to see the existence of
$(r_{k}^{\prime},s_{k}^{\prime})\in \triangle_{2}\cap Q_{N}$ such that
$\gamma_{1}^{-1}\cdot \tau_{k}\in F$ is the unique zero of $Z_{r_{k}^{\prime
},s_{k}^{\prime}}^{(2)}(\tau)$ in $F_{0}$. Therefore, together with (\ref{kl})
and (\ref{mo-3}), we conclude that%
\[
\sum_{\tau \in \mathbb{H}\backslash \{i,\rho \}}\nu_{\tau}(M_{N})\geq2|Q_{2N}%
\cap \triangle_{3}|+2=\frac{3N^{2}-8N+5}{4}+2,
\]
which is a contradiction with (\ref{fc-23}). This proves Step 3 and hence the proof
of Theorem \ref{thm20} is complete.
\end{proof}

\section{Existence and uniqueness of $\tau(C)$}

\label{tauc}

The purpose of this section is to prove the existence and uniqueness of $\tau(C)$ for $C\in\mathbb{R}\setminus\{0,1\}$ by applying Theorem \ref{thm2}. Given $C\in
\mathbb{R}$, we define the holomorphic function $f_{C}(\tau)$ on $\mathbb{H}$ by
\begin{equation}
f_{C}(\tau):=12(C\eta_{1}(\tau)-\eta_{2}(\tau))^{2}-g_{2}(\tau)(C-\tau)^{2}.
\label{i-54}%
\end{equation}
By $\eta
_{2}=\tau \eta_{1}-2\pi i$, we see that $f_{C}(\tau)=0$ if and only if (\ref{c=tau}) holds. This $f_C(\tau)$ appears in the expression of solutions of certain Painlev\'{e} VI equation (cf. \cite{CKL2,CKLT}).
The following result proves the existence and uniqueness of $\tau(C)$ as zeros of $f_{C}(\tau)$. Recall the fundamental domain $F_{0}$ of $\Gamma_{0}(2)$:
\[
F_{0}=\{ \tau \in \mathbb{H}\;|\;0\leq \operatorname{Re}\tau \leq1,\text{ \ }%
|\tau-\tfrac{1}{2}|\geq \tfrac{1}{2}\}.
\]

\begin{theorem}[Zero structure of $f_{C}(\tau)$ in $F_{0}$]
\label{Unique-pole}{\ }
\begin{itemize}
\item[(1)] For any $C\in \mathbb{R}\backslash \{0,1\}$, $f_{C}(\tau)$
has a unique zero $\tau(C)$ in $F_{0}$. Furthermore, $\tau(C)\in \mathring{F}_{0}$.
\item[(2)] For $C\in \{0,1\}$, $f_{C}(\tau)$
has no zeros in $F_{0}$.
\end{itemize}
\end{theorem}

Recall the classical result (cf.
\cite{YB}) that
\begin{equation}\label{d-eta1}
\eta_{1}^{\prime}(\tau)=\frac{i}{2\pi}\left(\eta_{1}(\tau)^{2}-\tfrac{1}%
{12}g_{2}(\tau)\right)  .
\end{equation}
To prove Theorem \ref{Unique-pole}, first we need the following lemma.

\begin{lemma}\label{lemma-etai}
If $\tau=ib$ with $b>0$, then $\eta_{1}'(\tau)\neq 0$ and
\begin{equation}\label{g2eta1}
g_{2}(\tau)-12\eta_{1}(\tau)^2>0,
\end{equation}
\begin{equation}
12\eta_{2}(\tau)^{2}-\tau^{2}g_{2}(\tau)>0. \label{i-47}%
\end{equation}
\end{lemma}
\begin{proof}
Denote $q=e^{2\pi i\tau}$. Recall the $q$-expansion of
$\eta_{1}(\tau)$ (see (\ref{eisenstein})):
\begin{equation}
\eta_{1}(\tau)=\frac{\pi^{2}}{3}-8\pi^{2}\sum_{k=1}^{\infty}b_{k}q^{k},\quad \text{where } b_{k}=\sum_{1\leq d|k}d.
\label{ex-1}%
\end{equation}%
Let $\tau=ib$ with $b>0$. Then $q=e^{-2\pi b}$ and hence $\frac{d}{db}\eta_{1}(ib)>0$ for $b>0$. So $\eta_{1}'(\tau)\neq 0$ and (\ref{g2eta1}) follows from (\ref{d-eta1}).
To prove (\ref{i-47}), we use the following modular property (see (\ref{II-31-1-00})):
\begin{equation}\label{eta1i}\eta_{1}(\tfrac{-1}{\tau})=\tau\eta_{2}(\tau),
\;\;g_{2}(\tfrac{-1}{\tau})=\tau^{4}g_{2}(\tau).\end{equation}
It follows that
\[12\eta_{2}(\tau)^{2}-\tau^{2}g_{2}(\tau)
=\frac{1}{\tau^{2}}
\left[12\eta_{1}(\tfrac{-1}{\tau})^{2}-g_{2}(\tfrac{-1}{\tau})\right]>0,\]
i.e. (\ref{i-47}) holds.\end{proof}

\begin{lemma}
\label{lemma-10}For any $C\in \mathbb{R}\backslash \{0,1\}$, $f_{C}(\tau
)\not =0$ for $\tau \in \partial F_{0}\cap \mathbb{H}$.
\end{lemma}

\begin{proof}
Suppose $f_{C}(\tau)=0$ for some $\tau \in \partial F_{0}\cap \mathbb{H}$.

\textbf{Case 1.} $\tau \in i\mathbb{R}^{+}$.

Then it is known that $g_2(\tau)>0$, $\eta_1(\tau)\in \mathbb{R}$ (see e.g. Lemma \ref{lemma-etai}) and $\eta_2(\tau)=\tau\eta_1(\tau)-2\pi i\in i\mathbb{R}$. It follows from $f_{C}(\tau)=0$ and (\ref{i-54}) that
\[\frac{2\pi i}{\tau-C}=\eta_1(\tau)\pm\sqrt{g_2(\tau)/12}\in\mathbb{R},\]
a contradiction with our assumption $C\in\mathbb{R}\setminus \{0\}$.

\textbf{Case 2.} $|\tau-\frac{1}{2}|=\frac{1}{2}$.

Then $\tau^{\prime}=\frac{\tau}{1-\tau}\in i\mathbb{R}^{+}$. Define
$C^{\prime}:=\frac{C}{1-C}\in \mathbb{R}\setminus\{0\}$. By $g_{2}(\tau^{\prime})=(1-\tau)^{4}g_{2}(\tau)$
and%
\begin{equation}
\eta_{2}(\tau^{\prime})=(1-\tau)\eta_{2}(\tau),\text{ }\eta_{1}(\tau^{\prime
})=(1-\tau)(\eta_{1}(\tau)-\eta_{2}(\tau)), \label{c-5}%
\end{equation}
a straightforward computation leads to%
\[
f_{C^{\prime}}(\tau^{\prime})=\frac{(1-\tau)^{2}}{(1-C)^{2}}f_{C}(\tau)=0.
\]
Then we obtain a contradiction as Case 1.

\textbf{Case 3.} $\tau \in1+i\mathbb{R}^{+}$.

Then $\tau^{\prime}=\tau-1\in i\mathbb{R}^{+}$. Define $C^{\prime}:=C-1\in\mathbb{R}\setminus\{0\}$. By
using $g_{2}(\tau^{\prime})=g_{2}(\tau)$ and%
\begin{equation}
\eta_{1}(\tau^{\prime})=\eta_{1}(\tau),\text{ }\eta_{2}(\tau^{\prime}%
)=\eta_{2}(\tau)-\eta_{1}(\tau), \label{c-6}%
\end{equation}
we easily obtain $f_{C^{\prime}}(\tau^{\prime})=f_{C}(\tau)=0$, again a
contradiction as Case 1.

The proof is complete.
\end{proof}

Recall the pre-modular form $Z_{r,s}^{(2)}(\tau)$ in Section \ref{modularform}.
Now we study the precise relation between $Z_{r,s}^{(2)}(\tau)$ and $f_{C}(\tau)$. This is {\it the key point} of our whole idea.
Fix any $C\in \mathbb{R}$, and for $s\in(0,\frac{1}{4(1+|C|)^{2}})$ we define%
\[
F_{C,s}(\tau):=\frac{4(\tau-C)}{s}Z_{-Cs,s}^{(2)}(\tau).
\]

\begin{lemma}
\label{lemma-9}Letting $s\rightarrow0$, $F_{C,s}(\tau)$ converges to
$f_{C}(\tau)$ uniformly in any compact subset of $F_{0}=\bar{F}_{0}%
\cap \mathbb{H}$.
\end{lemma}

\begin{proof}
Denote $u=-Cs+s\tau=s(\tau-C)$ for convenience. Then $u\rightarrow0$ as
$s\rightarrow0$. Let $\tau \in K$ where $K$ is any compact
subset of $F_{0}$. Then $g_{2}(\tau)$ and $g_{3}(\tau)$ are uniformly bounded
for $\tau \in K$. So it follows from the Laurent series of $\zeta(\cdot|\tau)$ and $\wp(\cdot|\tau)$ that
\begin{equation}
\zeta(-Cs+s\tau|\tau)=\frac{1}{u}-\frac{g_{2}(\tau)}{60}u^{3}+O(|u|^{5}),
\label{i-50}%
\end{equation}%
\[
\wp(-Cs+s\tau|\tau)=\frac{1}{u^{2}}+\frac{g_{2}(\tau)}{20}u^{2}+O(|u|^{4}),
\]%
\[
\wp^{\prime}(-Cs+s\tau|\tau)=\frac{-2}{u^{3}}+\frac{g_{2}(\tau)}%
{10}u+O(|u|^{3}),
\]
hold uniformly for $\tau \in K$ as $s\rightarrow0$. It follows from (\ref{c-7}) that%
\begin{equation}
Z_{-Cs,s}(\tau)=\frac{1}{u}+2\pi is-\eta_{1}u-\frac{g_{2}}{60}u^{3}+O(|u|^{5})
\label{i-51}%
\end{equation}
and so%
\begin{align}
Z_{-Cs,s}^{(2)}(\tau)  &  =Z_{-Cs,s}(\tau)^{3}-3\wp(-Cs+s\tau|\tau
)Z_{-Cs,s}(\tau)-\wp^{\prime}(-Cs+s\tau|\tau)\nonumber \\
&  =\frac{-12\pi^{2}s}{\tau-C}-12\pi i\eta_{1}s+3\eta_{1}^{2}u-\frac{g_{2}}%
{4}u+O(|u|^{2}) \label{i-52}%
\end{align}
uniformly for $\tau \in K$ as $s\rightarrow0$. Consequently, we derive from $u=s(\tau-C)$ and
$\eta_{2}=\tau \eta_{1}-2\pi i$ that
{\allowdisplaybreaks
\begin{align*}
F_{C,s}(\tau)  &  =\frac{4(\tau-C)}{s}Z_{-Cs,s}^{(2)}(\tau)\\
&  =-48\pi^{2}-48\pi i\eta_{1}(\tau-C)+12\eta_{1}^{2}(\tau-C)^{2}-g_{2}%
(\tau-C)^{2}+O(s)\\
&  =12((\tau-C)\eta_{1}-2\pi i)^{2}-g_{2}(\tau-C)^{2}+O(s)\\
&  =12(C\eta_{1}-\eta_{2})^{2}-g_{2}(\tau-C)^{2}+O(s)\rightarrow f_{C}(\tau)
\end{align*}
}%
uniformly for $\tau \in K$ as $s\rightarrow0$. The proof is complete.
\end{proof}

\begin{lemma}
\label{lemma-7}Let $s>0$. Then as $s\rightarrow0$, any zero $\tau(s)\in \{
\tau \in \mathbb{H}|\operatorname{Re}\tau \in \lbrack-1,1]\}$ of $Z_{-Cs,s}^{(2)}(\tau)$ (if exist) is uniformly bounded.
\end{lemma}

\begin{proof}
Suppose by contradiction that up to a subsequence of $s\rightarrow0$,
$Z_{-Cs,s}^{(2)}(\tau)$ has a zero $\tau(s)\in \{ \tau \in \mathbb{H}%
|\operatorname{Re}\tau \in \lbrack-1,1]\}$ such that $\tau(s)\rightarrow \infty$
as $s\rightarrow0$. Write $\tau=\tau(s)=a(s)+ib(s)$, then $a(s)\in
\lbrack-1,1]$ and $b(s)\rightarrow+\infty$.

Denote $q=e^{2\pi i\tau}$ as before. We recall the $q$-expansions (cf.
\cite[p.46]{Lang} for $\wp$ and \cite[(5.3)]{CKLW} for $Z_{r,s}$): for $|q|<|e^{2\pi iz}|<|q|^{-1}$,%
\begin{equation}
\frac{\wp(z|\tau)}{-4\pi^{2}}=\frac{1}{12}+\frac{e^{2\pi iz}}{(1-e^{2\pi
iz})^{2}}+\sum_{m=1}^{\infty}\sum_{n=1}^{\infty}nq^{nm}(e^{2\pi inz}+e^{-2\pi
inz}-2), \label{q-e1}%
\end{equation}%
{\allowdisplaybreaks
\begin{align*}
\frac{\wp^{\prime}(z|\tau)}{-4\pi^{2}} = &\frac{2\pi ie^{2\pi iz}%
}{(1-e^{2\pi iz})^{2}}+\frac{4\pi ie^{4\pi iz}}{(1-e^{2\pi iz})^{3}}\\
&  +2\pi i\sum_{m=1}^{\infty}\sum_{n=1}^{\infty}n^{2}q^{nm}(e^{2\pi
inz}-e^{-2\pi inz}),
\end{align*}
}%
\begin{align}
Z_{r,s}(\tau)  =&2\pi is-\pi i\frac{1+e^{2\pi iz}}{1-e^{2\pi iz}}\nonumber \\
&  -2\pi i\sum_{n=1}^{\infty}\left(  \frac{e^{2\pi iz}q^{n}}{1-e^{2\pi
iz}q^{n}}-\frac{e^{-2\pi iz}q^{n}}{1-e^{-2\pi iz}q^{n}}\right)  , \label{zz}%
\end{align}
where $z=r+s\tau$ in (\ref{zz}). We will also use the $q$-expansion of $g_{2}(\tau)$ (cf. \cite[p.44]{Lang}):
\begin{align}
g_{2}(\tau)  =\frac{4}{3}\pi^{4}+320\pi^{4}\sum_{k=1}^{\infty}\sigma
_{3}(k)q^{k},\;\,\text{where }\sigma_{3}(k)=\sum_{1\leq d|k}d^{3}.\label{ex-2}
\end{align}

Now we let $z=-Cs+s\tau=s(a(s)-C+ib(s))$ and denote $x=e^{2\pi iz}$ for
convenience. Then%
\[
e^{2\pi b(s)}=|q|^{-1}>|x|=e^{-2\pi sb(s)}>|q|=e^{-2\pi b(s)},
\]
so we can apply the above $q$-expansions. Notice that $|x|\in(0,1)$ and
$|x^{-1}q|=e^{-2\pi(1-s)b(s)}\rightarrow0$ as $s\rightarrow0$. There are two cases.

\textbf{Case 1.} Up to a subsequence $|x^{-1}q|=o(s|1-x|^{2})$.

Then we derive from (\ref{q-e1})-(\ref{zz}) that
\begin{equation}
\frac{\wp(z|\tau)}{-4\pi^{2}}=\frac{1}{12}+\frac{x}{(1-x)^{2}}+o(s|1-x|^{2}),
\label{zz-1}%
\end{equation}%
\[
\frac{\wp^{\prime}(z|\tau)}{-4\pi^{2}}=\frac{2\pi ix}{(1-x)^{2}}+\frac{4\pi
ix^{2}}{(1-x)^{3}}+o(s|1-x|^{2}),
\]%
\begin{equation}
Z_{-Cs,s}(\tau)=-\pi i\frac{1+x}{1-x}+2\pi is+o(s|1-x|^{2}). \label{zz-10}%
\end{equation}
We note that if $x\rightarrow1$, then $sb(s)\rightarrow0$ and%
\begin{align*}
1-x  &  =1-e^{2\pi is(a(s)-C+ib(s))}\\
&  =-2\pi is(a(s)-C+ib(s))+o(sb(s)).
\end{align*}
Together with $b(s)\rightarrow+\infty$ as $s\rightarrow0$, we always have
\begin{equation}
\frac{s^{2}}{1-x}=o(s). \label{zz-2}%
\end{equation}
Then by (\ref{zz-1})--(\ref{zz-10}) and (\ref{zz-2}), a straightforward
computation gives%
\begin{align*}
0  &  =Z_{-Cs,s}^{(2)}(\tau(s))=Z_{-Cs,s}(\tau)^{3}-3\wp(z|\tau)Z_{-Cs,s}%
(\tau)-\wp^{\prime}(z|\tau)\\
&  =-4\pi^{3}is+o(s),
\end{align*}
which is a contradiction.

\textbf{Case 2.} Up to a subsequence $|x^{-1}q|\geq ds|1-x|^{2}$ for some
constant $d>0$.

Then we see from (\ref{zz-2}) that%
\[
e^{-2\pi(1-s)b(s)}=|x^{-1}q|\geq ds|1-x|^{2}\geq s^{3}%
\]
and so $b(s)\leq \ln \frac{1}{s}$ for $s>0$ small. Then $u:=s(\tau
-C)=s(a(s)-C+ib(s))\rightarrow0$ and
\[
s=o(|u|),\text{ \ }u^{2}=o(s)\text{.}%
\]
Recall $q=e^{2\pi i\tau}$. Since $b(s)\rightarrow+\infty$, (\ref{ex-1}%
) and (\ref{ex-2}) show that
\[
g_{2}(\tau)=\frac{4}{3}\pi^{4}+O(|q|),\text{ \ }\eta_{1}(\tau)=\frac{1}{3}%
\pi^{2}+O(|q|)
\]
are uniformly bounded, so (\ref{i-50})-(\ref{i-52}) still hold, namely%
\begin{align*}
0   =Z_{-Cs,s}^{(2)}(\tau(s))
  =-12\pi i\eta_{1}s-\frac{12\pi^{2}s}{\tau-C}+3\eta_{1}^{2}u-\frac{g_{2}}%
{4}u+O(|u|^{2}).
\end{align*}
Since%
\[
3\eta_{1}^{2}-\frac{g_{2}}{4}=O(|q|)=O(e^{-2\pi b(s)})=O\left(  |\tau
-C|^{-2}\right)  ,
\]
we have
\[
\frac{-12\pi^{2}s}{\tau-C}+ 3\eta_{1}^{2}u-\frac{g_{2}}{4}u=O\left(\tfrac{s}{|\tau
-C|}\right)=o(s).
\]
Therefore, we finally obtain%
\begin{align*}
0  &  =-12\pi i\eta_{1}s-\frac{12\pi^{2}s}{\tau-C}+3\eta_{1}^{2}u-\frac{g_{2}%
}{4}u+O(|u|^{2})\\
&  =-4\pi^{3}is+o(s),
\end{align*}
which is a contradiction. The proof is complete.
\end{proof}

Recall $\triangle_{k}, k=0,1,2,3$, defined in (\ref{rectangle}). We define%
\begin{align*}
\tilde{\triangle}_{1}:=&   \{  (r,s)| (r+1,s)\in \triangle_{1} \}\\
=&\left \{  (r,s)|0<s<\tfrac{1}{2},\tfrac{-1}{2}<r<0,r+s>0\right \},
\end{align*}
\begin{align*}
\tilde{\triangle}_{2}:=& \{  (r,s)| (r+1,s)\in \triangle_{2} \}\\
=&\left \{  (r,s)|0<s<\tfrac{1}{2},\tfrac{-1}{2}<r<0,r+s<0\right \}.
\end{align*}
Since property (i) in Section \ref{modularform} gives $Z_{r,s}^{(2)}(\tau)=Z_{r+1,s}^{(2)}(\tau)$, we see
from Theorem \ref{thm2} that%
\begin{equation}
\text{For }(r,s)\in \tilde{\triangle}_{1}\cup \tilde{\triangle}_{2}\cup\triangle_{3}\text{, }%
Z_{r,s}^{(2)}(\tau)\text{ has a unique zero in }F_{0}\text{.} \label{cc-13}%
\end{equation}

From now on, we fix $C\in(-\infty,0)\cup(0,1)\cup(1,\infty)$. Then
$s\in(0,\frac{1}{4(1+|C|)^{2}})$ implies $(-Cs,s)\in \tilde{\triangle}_{1}\cup \tilde{\triangle}_{2}\cup\triangle_{3}$, so (\ref{cc-13}) implies that $Z_{-Cs,s}^{(2)}(\tau)$ has a
unique zero $\tau(s)\in F_{0}$. By the definition of $F_{0}$, we easily see
that%
\[
\frac{-1}{\tau(s)},\frac{\tau(s)}{1-\tau(s)}\in \{ \tau \in \mathbb{H}%
|\operatorname{Re}\tau \in \lbrack-1,1]\}.
\]

\begin{lemma}
\label{lemma-8}As $s\rightarrow0$, the unique zero $\tau(s)\in F_{0}$ of
$Z_{-Cs,s}^{(2)}(\tau)$ can not converge to any of $\{0,1,\infty \}$.
\end{lemma}

\begin{proof}
Lemma \ref{lemma-7} shows that $\tau(s)\not \rightarrow \infty$. To prove
$\tau(s)\not \rightarrow \{0,1\}$, we use the modular property (ii) of
$Z_{r,s}^{(2)}(\tau)$ in Section \ref{modularform}:%
\[
Z_{r^{\prime},s^{\prime}}^{(2)}(\tau^{\prime})=(c\tau+d)^{3}Z_{r,s}^{(2)}%
(\tau),
\]
whenever
\[
\tau^{\prime}=\frac{a\tau+b}{c\tau+d}\text{ and }(s^{\prime},r^{\prime})=(s,r)%
\begin{pmatrix}
d & -b\\
-c & a
\end{pmatrix}
,\text{ }%
\begin{pmatrix}
a & b\\
c & d
\end{pmatrix}
\in SL(2,\mathbb{Z}).
\]
We also use property (i):
\begin{equation}
Z_{m\pm r,n\pm s}^{(2)}(\tau)=\pm Z_{r,s}^{(2)}(\tau),\text{ }\forall
m,n\in \mathbb{Z}\text{.} \label{c-10}%
\end{equation}

Letting $%
\begin{pmatrix}
a & b\\
c & d
\end{pmatrix}
=%
\begin{pmatrix}
0 & -1\\
1 & 0
\end{pmatrix}
$, we obtain%
\begin{equation}
Z_{s,-r}^{(2)}(\tfrac{-1}{\tau})=\tau^{3}Z_{r,s}^{(2)}(\tau). \label{c-101}%
\end{equation}
Recall $C\in(-\infty,0)\cup(0,1)\cup(1,+\infty)$ and $s\in(0,\frac
{1}{4(1+|C|)^{2}})$.

\textbf{Case 1. }$C\in(-\infty,0)$.

By defining%
\[
\tilde{C}:=\frac{-1}{C},\text{ \  \ }\tilde{s}:=-Cs,
\]
we have $\tilde{s}\in(0,\frac{1}{4(1+|\tilde{C}|)^{2}})$ for $s$ small and%
\[
\tau^{3}Z_{-Cs,s}^{(2)}(\tau)=Z_{s,Cs}^{(2)}(\tfrac{-1}{\tau})=-Z_{-s,-Cs}%
^{(2)}(\tfrac{-1}{\tau})=-Z_{-\tilde{C}\tilde{s},\tilde{s}}^{(2)}(\tfrac
{-1}{\tau}).
\]
Therefore, $Z_{-\tilde{C}\tilde{s},\tilde{s}}^{(2)}(\tau)$ has zero $\frac
{-1}{\tau(s)}\in \{ \tau \in \mathbb{H}|\operatorname{Re}\tau \in \lbrack-1,1]\}$.
Since $\tilde{s}\rightarrow0$ as $s\rightarrow0$, Lemma \ref{lemma-7} implies
$\frac{-1}{\tau(s)}\not \rightarrow \infty$, i.e. $\tau(s)\not \rightarrow 0$
as $s\rightarrow0$.

\textbf{Case 2. }$C\in(0,1)\cup(1,+\infty)$.

By defining%
\[
\tilde{C}:=\frac{-1}{C},\text{ \  \ }\tilde{s}:=Cs,
\]
we have $\tilde{s}\in(0,\frac{1}{4(1+|\tilde{C}|)^{2}})$ for $s$ small and%
\[
\tau^{3}Z_{-Cs,s}^{(2)}(\tau)=Z_{s,Cs}^{(2)}(\tfrac{-1}{\tau})=Z_{-\tilde
{C}\tilde{s},\tilde{s}}^{(2)}(\tfrac{-1}{\tau}).
\]
Again we obtain $\tau(s)\not \rightarrow 0$ as $s\rightarrow0$.

Therefore, we have proved $\tau(s)\not \rightarrow 0$ as $s\rightarrow0$.

Finally, to prove $\tau(s)\not \rightarrow 1$ as $s\rightarrow0$, we let $%
\begin{pmatrix}
a & b\\
c & d
\end{pmatrix}
=%
\begin{pmatrix}
1 & 0\\
-1 & 1
\end{pmatrix}
$ and obtain
\begin{equation}
Z_{r,r+s}^{(2)}(\tfrac{\tau}{1-\tau})=(1-\tau)^{3}Z_{r,s}^{(2)}(\tau).
\label{c-102}%
\end{equation}

\textbf{Case 3. }$C\in(-\infty,0)\cup(0,1)$.

By defining%
\[
\tilde{C}:=\frac{C}{1-C},\text{ \  \ }\tilde{s}:=(1-C)s,
\]
we have $\tilde{s}\in(0,\frac{1}{4(1+|\tilde{C}|)^{2}})$ for $s$ small and%
\[
(1-\tau)^{3}Z_{-Cs,s}^{(2)}(\tau)=Z_{-Cs,(1-C)s}^{(2)}(\tfrac{\tau}{1-\tau
})=Z_{-\tilde{C}\tilde{s},\tilde{s}}^{(2)}(\tfrac{\tau}{1-\tau}).
\]
So $Z_{-\tilde{C}\tilde{s},\tilde{s}}^{(2)}(\tau)$ has zero $\frac{\tau
(s)}{1-\tau(s)}\in \{ \tau \in \mathbb{H}|\operatorname{Re}\tau \in \lbrack
-1,1]\}$, and Lemma \ref{lemma-7} implies $\frac{\tau(s)}{1-\tau
(s)}\not \rightarrow \infty$, i.e. $\tau(s)\not \rightarrow 1$ as
$s\rightarrow0$.

\textbf{Case 4. }$C\in(1,+\infty)$.

By defining%
\[
\tilde{C}:=\frac{C}{1-C},\text{ \  \ }\tilde{s}:=-(1-C)s,
\]
we have $\tilde{s}\in(0,\frac{1}{4(1+|\tilde{C}|)^{2}})$ for $s$ small and%
\begin{align*}
(1-\tau)^{3}Z_{-Cs,s}^{(2)}(\tau)  &  =Z_{-Cs,(1-C)s}^{(2)}(\tfrac{\tau
}{1-\tau})\\
&  =-Z_{Cs,-(1-C)s}^{(2)}(\tfrac{\tau}{1-\tau})=-Z_{-\tilde{C}\tilde{s}%
,\tilde{s}}^{(2)}(\tfrac{\tau}{1-\tau}).
\end{align*}
Again we obtain $\tau(s)\not \rightarrow 1$ as $s\rightarrow0$.

The proof is complete.
\end{proof}

Now we are in a position to prove Theorem \ref{Unique-pole}-(1).

\begin{proof}
[Proof of Theorem \ref{Unique-pole}-(1)]Fix $C\in \mathbb{R}\backslash \{0,1\}$ and
let $s\in(0,\frac{1}{4(1+|C|)^{2}})$. Recall that $\tau(s)$ is the unique zero
of $Z_{-Cs,s}^{(2)}(\tau)$ in $F_{0}$. By Lemma \ref{lemma-8}, up to a
subsequence of $s\to 0$, we have
\begin{equation}\label{limit}
\tau(C):=\lim_{s\rightarrow0}\tau(s)\in \bar{F}_{0}\cap \mathbb{H=}F_{0}.
\end{equation}
Recalling%
\[
F_{C,s}(\tau)=\frac{4(\tau-C)}{s}Z_{-Cs,s}^{(2)}(\tau),
\]
we have $F_{C,s}(\tau(s))=0$. Then Lemma \ref{lemma-9} implies $f_{C}%
(\tau(C))=0$, namely $f_{C}(\tau)$ has a zero $\tau(C)\in F_{0}$. Applying
Lemma \ref{lemma-10}, we have $\tau(C)\in \mathring{F}_{0}$. Suppose
$f_{C}(\tau)$ has another zero $\tau_{1}\not =\tau(C)$ in $\mathring{F}_{0}$.
Since $F_{C,s}(\tau)$ and $f_{C}(\tau)$ are all holomorphic functions, it
follows from Lemma \ref{lemma-9} and Rouch\'{e}'s theorem that
$F_{C,s}(\tau)$ has a zero $\tau_{1}(s)$ satisfying $\tau_{1}(s)\rightarrow
\tau_{1}$ as $s\rightarrow0$, namely $Z_{-Cs,s}^{(2)}(\tau)$ has two different
zeros $\tau(s)$ and $\tau_{1}(s)$ in $F_{0}$ when $s>0$ small, a contradiction
with (\ref{cc-13}). Therefore, $\tau(C)$ is the unique zero of $f_{C}(\tau)$ in
$F_{0}$. This also implies that (\ref{limit}) actually holds for $s\to 0$ (i.e. not only for a subsequence). The proof is complete.
\end{proof}

The proof of Theorem \ref{Unique-pole}-(2) will be postponed in the next section.
As in Theorem \ref{Unique-pole}-(1), we always denote by $\tau(C)$
the unique zero of $f_{C}(\tau)$ in $F_{0}$. Since we proved in \cite[Theorem 1.1]{CKL2} that for fixed $C$,
\begin{equation}\label{simplezero}\text{\emph{$f_{C}(\tau)$ has at most simple zeros on $\mathbb{H}$},}\end{equation} the implicit function theorem infers that $\tau(C)$ is a \emph{smooth}
function of $C\in \mathbb{R}\backslash \{0,1\}$.  We conclude this section by proving basic properties of $\tau(C)$.

\begin{lemma}
\label{lemma-13}The smooth function $\tau(C)$ satisfies%
\begin{equation}
\tau \left(  \frac{1}{1-C}\right)  =\frac{1}{1-\tau(C)},\text{ \ }\forall
C\in \mathbb{R}\backslash \{0,1\}. \label{i-65}%
\end{equation}

\end{lemma}

\begin{proof}
Let $\tau^{\prime}=\frac{1}{1-\tau}$ and $C^{\prime}=\frac{1}{1-C}$, then it
easy to prove that%
\begin{equation}
\tau^{\prime}\in F_{0}\Longleftrightarrow \tau \in F_{0}\text{ and }C^{\prime
}\in \mathbb{R}\backslash \{0,1\} \Longleftrightarrow C\in \mathbb{R}%
\backslash \{0,1\}. \label{c-15}%
\end{equation}
By using $g_{2}(\tau^{\prime})=(1-\tau)^{4}g_{2}(\tau)$ and%
\begin{equation}
\eta_{2}(\tau^{\prime})=(1-\tau)\eta_{1}(\tau),\text{ }\eta_{1}(\tau^{\prime
})=(1-\tau)(\eta_{1}(\tau)-\eta_{2}(\tau)), \label{c-16}%
\end{equation}
a straightforward computation leads to%
\[
f_{C^{\prime}}(\tau^{\prime})=\frac{(1-\tau)^{2}}{(1-C)^{2}}f_{C}(\tau).
\]
So $f_{C}(\tau(C))=0$ gives $f_{C^{\prime}}(\frac{1}{1-\tau(C)})=0$. Applying
Theorem \ref{Unique-pole}-(1), we obtain (\ref{i-65}). This completes the proof.
\end{proof}

\begin{lemma}
\label{lemma-11}Write $\tau(C)=a(C)+b(C)i$ with $a(C),b(C)\in \mathbb{R}$. Then%
\begin{equation}
b(C)\rightarrow+\infty,\text{ }a(C)\rightarrow \left \{
\begin{array}
[c]{c}%
1/4\text{ \ if }C\rightarrow+\infty,\\
3/4\text{ \ if }C\rightarrow-\infty,
\end{array}
\right.  \label{i-63}%
\end{equation}%
\begin{equation}
\tau(C)\rightarrow0\text{ as }C\rightarrow0\text{ \ and \ }\tau(C)\rightarrow
1\text{ as }C\rightarrow1\text{.} \label{i-64}%
\end{equation}

\end{lemma}

\begin{proof}
Recalling (\ref{c=tau}), we define
\begin{equation}
\phi_{\pm}(\tau):=\tau-\frac{2\pi i}{\eta_{1}(\tau)\pm \sqrt{g_{2}(\tau)/12}%
},\text{ \ }\tau \in F_{0}. \label{i-53}%
\end{equation}
Write $\tau=a+bi$ and $q=e^{2\pi i\tau}$ as before. Recall from the $q$-expansions (\ref{ex-1}) and (\ref{ex-2}) that
\begin{equation}
\eta_{1}(\tau)=\frac{1}{3}\pi^{2}-8\pi^{2}(q+3q^{2})+O(|q|^{3}), \label{c-30}%
\end{equation}%
\begin{equation}
g_{2}(\tau)=\frac{4}{3}\pi^{4}+320\pi^{4}(q+9q^{2})+O(|q|^{3}). \label{c-31}%
\end{equation}
For $\tau\in F_{0}$, we fix the branch of $\sqrt{g_{2}(\tau)/12}$ near $b=+\infty$ such that
$\sqrt{g_{2}(\tau)/12}=\frac{1}{3}\pi^{2}+O(|q|)$ near $b=+\infty$. Then we
easily obtain%
\[
\eta_{1}(\tau)-\sqrt{g_{2}(\tau)/12}=-48\pi^{2}q\left(  1-42q+O(|q|^{2}%
)\right)  ,
\]
and so%
\begin{align*}
\phi_{-}(\tau)  &  =\tau+\frac{i}{24\pi}q^{-1}+\frac{7i}{4\pi}+O(|q|)\\
&  =a+\frac{\sin2\pi a}{24\pi}e^{2\pi b}+i\bigg(  b+\frac{\cos2\pi a}{24\pi
}e^{2\pi b}+\frac{7}{4\pi}\bigg)  +O(|q|).
\end{align*}
Therefore, when $C\in \mathbb{R}$ and $|C|\rightarrow+\infty$, it is easy to
prove the existence of $\tau_{1}(C)=a_{1}(C)+ib_{1}(C)\in \mathring{F}_{0}$
such that $C=\phi_{-}(\tau_{1}(C))$ and%
\[
b_{1}(C)\rightarrow+\infty,\text{ }a_{1}(C)\rightarrow \left \{
\begin{array}
[c]{c}%
1/4\text{ \ if }C\rightarrow+\infty,\\
3/4\text{ \ if }C\rightarrow-\infty,
\end{array}
\right.
\]
i.e. $\tau_{1}(C)\rightarrow \infty$ as $C\rightarrow \pm \infty$. By $\eta
_{2}=\tau \eta_{1}-2\pi i$, (\ref{i-53}) and (\ref{i-54}), it follows
that $C=\phi_{-}(\tau_{1}(C))$ implies $f_{C}(\tau_{1}(C))=0$. Since $\tau(C)$
is the unique zero of $f_{C}$ in $F_{0}$, we conclude $\tau(C)=\tau_{1}(C)$.
This proves (\ref{i-63}). Finally, (\ref{i-64}) follows from (\ref{i-63}) and
(\ref{i-65}). The proof is complete.
\end{proof}

\section{Critical points of $\eta_{1}(\tau)$ or equivalently $E_2(\tau)$}

\label{NOZ}

\subsection{Location of critical points of $\eta_{1}(\tau)$}
This section is devoted to the proof of our main results.
Note that $e^{\pi i/3}=\frac{1}{1-e^{\pi i/3}}$. Then by
\[\eta_1(\tfrac{1}{1-\tau})=(1-\tau)(\eta_{1}(\tau)-\eta_{2}(\tau))\]
and the Legendre relation $\eta_{2}(\tau)=\tau\eta_{1}(\tau)-2\pi i$, we easily obtain
\begin{equation}\label{eta1rho}\eta_{1}(e^{\pi i/3})={2\pi}/{\sqrt{3}}.\end{equation}
First we prove the following result, which implies Theorems \ref{eta-thm0}-\ref{F-one} as consequences.

\begin{theorem}
\label{Coro-1} Let $\tau(C)$ be the unique zero of $f_{C}(\tau)$ for $C\in \mathbb{R}\backslash \{0,1\}$ in Theorem \ref{Unique-pole}-(1). Then
the followings hold:

\begin{itemize}
\item[(1)] For any $m\in \mathbb{Z}$, there holds $\eta_{1}^{\prime}%
(\tau)\not =0$ in $F_{0}+m$. Consequently, $\eta_{1}^{\prime}(\tau)\not =0$
whenever $\operatorname{Im}\tau \geq \frac{1}{2}$.

\item[(2)] Given $\gamma=%
\begin{pmatrix}
a & b\\
c & d
\end{pmatrix}
\in \Gamma_{0}(2)/\{ \pm I_{2}\}$ with $c\not =0$. Then $\frac{a\tau
(-d/c)+b}{c\tau(-d/c)+d}$ is the unique zero of $\eta_{1}^{\prime}(\tau)$ in
the fundamental domain $\gamma(F_{0})$ of $\Gamma_{0}(2)$. In particular,%
\begin{equation}
\Theta:=\left \{  \left.  \frac{a\tau(\tfrac{-d}{c})+b}{c\tau(\tfrac{-d}{c}%
)+d}\right \vert
\begin{pmatrix}
a & b\\
c & d
\end{pmatrix}
\in \Gamma_{0}(2)/\{ \pm I_{2}\} \text{ with }c\not =0\right \}
\label{zero-set}%
\end{equation}
gives rise to all the zeros of $\eta_{1}^{\prime}(\tau)$ in $\mathbb{H}$.
\end{itemize}
\end{theorem}

\begin{proof}(1). First we claim that%
\begin{equation}
12\eta_{1}(\tau)^{2}-g_{2}(\tau)\not =0\text{ for }\tau \in \partial F_{0}%
\cap \mathbb{H}. \label{i-55}%
\end{equation}

If $\tau \in i\mathbb{R}^{+}$, $12\eta_{1}(\tau)^{2}-g_{2}(\tau)\not =0$
follows from Lemma \ref{lemma-etai}. If $\tau \in i\mathbb{R}^{+}+1$, then%
\[
12\eta_{1}(\tau)^{2}-g_{2}(\tau)=12\eta_{1}(\tau-1)^{2}-g_{2}(\tau-1)\not =0.
\]

If $|\tau-\frac{1}{2}|=\frac{1}{2}$, then $\tau^{\prime}=\frac{\tau}{1-\tau
}\in i\mathbb{R}^{+}$. By using (\ref{c-5}), we see from (\ref{i-54}) with
$C=-1$ that%
\begin{align*}
f_{-1}(\tau^{\prime})  &  =12(\eta_{1}(\tau^{\prime})+\eta_{2}(\tau^{\prime
}))^{2}-g_{2}(\tau^{\prime})(1+\tau^{\prime})^{2}\\
&  =(1-\tau)^{2}[12\eta_{1}(\tau)^{2}-g_{2}(\tau)].
\end{align*}
Since Lemma \ref{lemma-10} shows $f_{-1}(\tau^{\prime})\not =0$, we obtain
$12\eta_{1}(\tau)^{2}-g_{2}(\tau)\not =0$. This proves (\ref{i-55}).

Suppose by contradiction that $12\eta_{1}(\tau)^{2}-g_{2}(\tau)$ has a zero
$\tau_{0}$ in $\mathring{F}_{0}$. Then%
\[
\text{either }\eta_{1}(\tau_{0})-\sqrt{g_{2}(\tau_{0})/12}=0\text{ or }%
\eta_{1}(\tau_{0})+\sqrt{g_{2}(\tau_{0})/12}=0.
\]
Without loss of generality, we may assume $\eta_{1}(\tau_{0})+\sqrt{g_{2}%
(\tau_{0})/12}=0$. Recall (\ref{eta1rho}) and the fact that $g_{2}(\tau)=0$ in $F_{0}$ if and only if $\tau=e^{\pi i/3}$. So $\tau
_{0}\not =e^{\pi i/3}$ and $g_{2}(\tau_{0})\not =0$, i.e. $\eta_{1}(\tau)+\sqrt{g_{2}%
(\tau)/12}$ is holomorphic at $\tau_0$ with $\tau_0$ being a zero. Recalling $\phi_{\pm
}(\tau)$ in (\ref{i-53}), it follows that $\phi_{+}(\tau)$ is meromorphic at
$\tau_{0}$ with $\tau_{0}$ being a pole and so maps a small neighborhood
$U\subset \mathring{F}_{0}$ of $\tau_{0}$ onto a neighborhood of $\infty$.
Then for $C>0$ large enough, there exists $\tau_{1}(C)\in
U$ such that $C=\phi_{+}(\tau_{1}(C))$, which implies $f_{C}(\tau_{1}(C))=0$.
Applying Theorem \ref{Unique-pole}-(1) and Lemma \ref{lemma-11}, we
obtain $\tau_{1}(C)=\tau(C)\rightarrow \infty$ as $C\rightarrow+\infty$, which
contradicts with $\tau_{1}(C)\in U$.

Therefore, we have proved that \begin{equation}\label{eta1g2}12\eta_{1}(\tau)^{2}-g_{2}(\tau)\neq 0\quad\text{for any }\tau\in F_{0}.\end{equation}
Since%
\[
\eta_{1}^{\prime}(\tau)=\frac{i}{2\pi}\left(  \eta_{1}(\tau)^{2}-\tfrac{1}{12}g_{2}(\tau)\right)
\]
and $\eta_{1}(\tau+1)=\eta_{1}(\tau)$, we conclude that $\eta_{1}^{\prime
}(\tau)\not =0$ for any $\tau \in F_{0}+m$ and $m\in \mathbb{Z}$. This proves (1).

(2). Given $\gamma=%
\begin{pmatrix}
a & b\\
c & d
\end{pmatrix}
\in \Gamma_{0}(2)/\{\pm I_2\}$ with $c\not =0$. Write $\tau^{\prime}%
=\gamma \cdot \tau=\frac{a\tau+b}{c\tau+d}$ with $\tau \in F_{0}$. By using%
\begin{equation}
\eta_{1}(\tau^{\prime})=(c\tau+d)(c\eta_{2}(\tau)+d\eta_{1}(\tau)),\text{
\ }g_{2}(\tau^{\prime})=(c\tau+d)^{4}g_{2}(\tau), \label{c-38}%
\end{equation}
we have%
\begin{align*}
12\eta_{1}(\tau^{\prime})^{2}-g_{2}(\tau^{\prime})  &  =c^{2}(c\tau
+d)^{2}\left[  12(\tfrac{d}{c}\eta_{1}+\eta_{2})^{2}-g_{2}(\tau)(\tfrac{d}%
{c}+\tau)^{2}\right] \\
&  =c^{2}(c\tau+d)^{2}f_{\tfrac{-d}{c}}(\tau).
\end{align*}
Clearly $\tfrac{-d}{c}\in \mathbb{Q}\backslash \mathbb{Z}$, so Theorem
\ref{Unique-pole}-(1) shows that $\tau(\tfrac{-d}{c})\in \mathring{F}_{0}$ is the unique zero of
$f_{\tfrac{-d}{c}}(\tau)$ in $F_{0}$. Consequently,
\[
\gamma \cdot \tau(\tfrac{-d}{c})=\frac{a\tau(\tfrac{-d}{c})+b}{c\tau(\tfrac
{-d}{c})+d}\in \gamma(\mathring{F}_{0})
\]
is the unique zero of $12\eta_{1}^{2}-g_{2}$ in $\gamma(F_{0})$. Since%
\[
\mathbb{H=}\bigcup_{\gamma \in \Gamma_{0}(2)/\{ \pm I_{2}\}}\gamma (F_{0}),
\]
we conclude that the set $\Theta$ defined in (\ref{zero-set}) gives all the
zeros of $12\eta_{1}^{2}-g_{2}$ and so $\eta_{1}^{\prime}$. This proves (2). The proof is complete.
\end{proof}

Recall the curves defined in Section 1:
\[
\mathcal{C}_{-}=\{ \tau(C)|C\in(-\infty,0)\},\text{ \ }\mathcal{C}_{+}=\{
\tau(C)|C\in(1,+\infty)\},
\]%
\[
\mathcal{C}_{0}=\{ \tau(C)|C\in(0,1)\}.
\]

\begin{proof}[Proof of Theorem \ref{Location}]
The smoothness of the three curves will be given in Section \ref{geometricinter}.
The assertion that under the M\"{o}bius transformation of $\Gamma_{0}(2)$ action, the collection of all critical points of $\eta_{1}(\tau)$ is precisely the set $\mathcal{C}$ given by (\ref{cf0}), is a direct consequence of the expression (\ref{zero-set}) of the critical point set $\Theta$. Recall that $\tau(C)$ is smooth as a function of $C\in\mathbb{R}\setminus\{0,1\}$.
To prove the denseness, i.e. the identity (\ref{cf1}), it suffices to prove that
\[Q_0:=\{\tfrac{-d}{c}\;|\;d\in\mathbb{Z},\, c\in 2\mathbb{Z}\setminus\{0\}, \, (c,d)=1\}\]
is dense in $\mathbb{Q}$ and hence dense in $\mathbb{R}$. Take any $\frac{m}{n}\in \mathbb{Q}\setminus\{0\}$ such that $m,n\in\mathbb{Z}$ and $(m,n)=1$.

{\bf Case 1.} $n$ is even. Then $\frac{m}{n}\in Q_0$.

{\bf Case 2.} $n$ is odd and $m$ is odd. Then
\[Q_0\ni\frac{m(2^k+n)}{2^kn}\to \frac{m}{n}\quad \text{as}\;k\to+\infty.\]

{\bf Case 3.} $n$ is odd and $m$ is even. Then
\[Q_0\ni\frac{2^k m+n}{2^kn}\to \frac{m}{n}\quad \text{as}\;k\to+\infty.\]
This proves $\overline{Q_{0}}=\mathbb{R}$ and so completes the proof.
\end{proof}

\begin{proof}[Proof of Corollary \ref{line1/2}]
Let $\gamma=
\begin{pmatrix}
1 & -1\\
2 & -1
\end{pmatrix}
\in \Gamma_{0}(2)$. Then it is easy to prove that%
\[
\gamma (F_{0})=\left \{  \tau \in \mathbb{H}\,|\,|\tau-\tfrac{1}{2}|\leq \tfrac
{1}{2},\text{ }|\tau-\tfrac{1}{4}|\geq \tfrac{1}{4},\text{ }|\tau-\tfrac{3}%
{4}|\geq \tfrac{1}{4}\right \}  ,
\]
and so%
\[
\left \{  \tau \in \mathbb{H}|\operatorname{Re}\tau=\tfrac{1}{2}\right \}  \subset
F_{0}\cup \gamma (F_{0}).
\]
Applying Theorem \ref{Coro-1}-(2), we see that $\tilde{\tau}:=\frac{\tau(1/2)-1}{2\tau(1/2)-1}$ is
the unique zero of $\eta_{1}^{\prime}$ in $\gamma (F_{0})$. We will prove
in Theorem \ref{lemma-15} (iii) that $\tau(1/2)=\frac{1}{2}+i\hat{b}$ for some
$\hat{b}\in (\frac{\sqrt{3}}{2},\frac{6}{5})$. Then $
\tilde{\tau}=\frac{1}{2}+\frac{i}{4\hat{b}}\in \{\tau
\in \mathbb{H}|\operatorname{Re}\tau=\tfrac{1}{2}\}$, namely $\tilde{\tau}$ is the unique zero of
$\eta_{1}^{\prime}$ on the line $\{ \tau \in \mathbb{H}|\operatorname{Re}%
\tau=\frac{1}{2}\}$ with $b_0:=\operatorname{Im}\tilde{\tau}\in(\frac{5}{24},\frac{1}%
{2\sqrt{3}})$. Recall (\ref{ex-1}) and (\ref{eta1rho}) that $\eta_1(\frac{1}{2}+i\frac{\sqrt{3}}{2})>\frac{\pi^2}{3}
=\lim_{b\to+\infty}\eta_1(\frac{1}{2}+ib)$. Moreover, it follows from (\ref{eisenstein}) and (\ref{sl-eta}) that for $\tau=\frac{1}{2}+ib$,
\[\eta_{1}\left(\tfrac{1}{2}+\tfrac{i}{4b}\right)=\eta_1(\tfrac{\tau-1}{2\tau-1})
=-4b^{2}\eta_1(\tau)+8\pi b\to -\infty\;\text{as} \;b\to+\infty. \]
Thus $\eta
_{1}(\frac12+ib)$ is strictly increasing for $b\in (0, b_0)$ and strictly decreasing for $b\in (b_0,+\infty)$. The proof is complete.
\end{proof}

Now we can finish the proof of Theorem \ref{Unique-pole}.
\begin{proof}
[Proof of Theorem \ref{Unique-pole}-(2)]First we consider $C=0$, i.e.%
\[
f_{0}(\tau)=12\eta_{2}(\tau)^{2}-g_{2}(\tau)\tau^{2}.
\]
Suppose $f_{0}(\tau)=0$ for some $\tau \in F_{0}$. Then it is easy to see that
$\tau^{\prime}:=\frac{\tau-1}{\tau}\in F_{0}$. By%
\[
\eta_{1}(\tau^{\prime})=\tau \eta_{2}(\tau),\text{ \ }g_{2}(\tau^{\prime}%
)=\tau^{4}g_{2}(\tau),
\]
we obtain%
\[
12\eta_{1}(\tau^{\prime})^{2}-g_{2}(\tau^{\prime})=\tau^{2}f_{0}(\tau)=0,
\]
a contradiction with (\ref{eta1g2}).

Now we consider $C=1$, i.e.%
\[
f_{1}(\tau):=12(\eta_{1}(\tau)-\eta_{2}(\tau))^{2}-g_{2}(\tau)(1-\tau)^{2}.
\]
Suppose $f_{1}(\tau)=0$ for some $\tau \in F_{0}$. Then it is easy to see that
$\tau^{\prime}:=\frac{1}{1-\tau}\in F_{0}$. By%
\[
\eta_{1}(\tau^{\prime})=(1-\tau)(\eta_{1}(\tau)-\eta_{2}(\tau)),\text{
\ }g_{2}(\tau^{\prime})=(1-\tau)^{4}g_{2}(\tau),
\]
we obtain%
\[
12\eta_{1}(\tau^{\prime})^{2}-g_{2}(\tau^{\prime})=(1-\tau)^{2}f_{1}(\tau)=0,
\]
again a contradiction with (\ref{eta1g2}). The proof is complete.
\end{proof}

\subsection{Geometry of curves $\mathcal{C}_{-}$, $\mathcal{C}_{0}$ and $\mathcal{C}_{+}$}
In this section, we want to prove some geometry about these three curves, including their intersection with the line $\operatorname{Re}\tau=\frac{1}{2}$.

\begin{theorem}\label{lemma-15} {\ }

\begin{itemize}
\item[(i)] The function $C\mapsto \tau(C)$ is one-to-one whenever $C$ is
restricted in one of $(-\infty,0)$, $(0,1)$ and $(1,+\infty)$, i.e. any one of
curves $\mathcal{C}_{-}$, $\mathcal{C}_{0}$, $\mathcal{C}_{+}$ has no self-intersection. Furthermore,
\[\partial \mathcal{C}_{0}=\{0,1\},\;\;\partial \mathcal{C}_{-}=\{0,\tfrac{3}{4}+i\infty\}, \;\;\partial \mathcal{C}_{+}=\{1,\tfrac{1}{4}+i\infty\}.\]

\item[(ii)] The curve $\mathcal{C}_{0}$ is symmetric with respect to the line
$\operatorname{Re}\tau=\frac{1}{2}$; $\mathcal{C}_{-}$ is symmetric with
$\mathcal{C}_{+}$ with respect to the line $\operatorname{Re}\tau=\frac{1}{2}$.

\item[(iii)] $\tau(\frac{1}{2})$ is the unique intersection point of the curve
$\mathcal{C}_{0}$ with the line $\operatorname{Re}\tau=\frac{1}{2}$.
Furthermore, $\operatorname{Im}\tau(\frac{1}{2})\in(\frac{\sqrt{3}}{2}%
,\frac{6}{5})$.

\item[(iv)] $\mathcal{C}_{-}$ (resp. $\mathcal{C}_{+}$) has a unique
intersection point $\tau_{-}$ with the line $\operatorname{Re}\tau=\frac{1}%
{2}$. Furthermore, $\operatorname{Im}\tau_{-}\in(\frac{1}{2},\frac{\sqrt{3}%
}{2})$.

\item[(v)] $\tau_{-}$ is the unique intersection point of $\mathcal{C}_{-}$
with $\mathcal{C}_{+}$.

\item[(vi)] $\frac{1}{1-\tau_{-}}$ (resp. $\frac{\tau_{-}}{\tau_{-}-1}$) is the
unique intersection point of $\mathcal{C}_{0}$ with $\mathcal{C}_{-}$ (resp.
$\mathcal{C}_{+}$).
\end{itemize}
\end{theorem}

First we prove that any one of these curves has no self-intersection.

\begin{lemma}
\label{lemma-14}The function $C\mapsto \tau(C)$ is one-to-one whenever $C$ is
restricted in one of $(-\infty,0)$, $(0,1)$ and $(1,+\infty)$.
\end{lemma}

\begin{proof}
We recall the following results (cf. \cite{LW}): when $\tau=\frac{1}{2}+ib$
with $b>0$,%
\begin{equation}
\eta_{1}(\tau),g_{2}(\tau)\in \mathbb{R}\text{ \  \ and \ }g_{2}(\tau)\left \{
\begin{array}
[c]{l}%
>0\text{ \ if \ }b>\frac{\sqrt{3}}{2},\\
=0\text{ \ if \ }b=\frac{\sqrt{3}}{2},\\
<0\text{ \ if \ }b\in \lbrack \frac{1}{2},\frac{\sqrt{3}}{2}).
\end{array}
\right.  \label{i-66}%
\end{equation}

As pointed out before, $f_{C}(\tau(C))=0$ is equivalent to%
\begin{equation}\label{phipm}
\text{either \ }C=\phi_{+}(\tau(C))=\tau(C)-\frac{2\pi i}{\eta_{1}%
(\tau(C))+\sqrt{g_{2}(\tau(C))/12}}%
\end{equation}
\[
\text{or \ }C=\phi_{-}(\tau(C))=\tau(C)-\frac{2\pi i}{\eta_{1}(\tau
(C))-\sqrt{g_{2}(\tau(C))/12}}.
\]
If $\tau(C)=\frac{1}{2}+ib(C)$ with $b(C)\geq \frac{\sqrt{3}}{2}$ for some
$C\in \mathbb{R}\backslash \{0,1\}$, then it follows from (\ref{i-66}) and (\ref{eta1g2}) that
$C=\phi_{\pm}(\tau(C))=\operatorname{Re}\tau(C)=\frac{1}{2}$. Therefore,%
\begin{equation}
\mathcal{C}_{-},\mathcal{C}_{+}\subset F_{0}\backslash \left \{  \tau=\tfrac
{1}{2}+ib\text{ }|\text{ }b\geq \tfrac{\sqrt{3}}{2}\right \}  . \label{i8}%
\end{equation}
Remark that $g_{2}(\tau)=0$ for $\tau\in F_{0}$ if and only if $\tau=\frac{1}{2}+\frac{\sqrt{3}}{2}i$. We restrict $\tau \in F_{0}\backslash \{ \tau=\tfrac{1}{2}+ib|
b\geq \tfrac{\sqrt{3}}{2}\}$ and fix a branch of $\sqrt{g_{2}(\tau)/12}$. Then it follows from (\ref{eta1g2}) that
both $\phi_{+}(\tau)$ and $\phi_{-}(\tau)$ are \emph{single-valued
holomorphic} functions in $F_{0}\backslash \{ \tau=\tfrac{1}{2}+ib$ $|$
$b\geq \tfrac{\sqrt{3}}{2}\}$. Define%
\[
D_{1}:=\{C\in(-\infty,0)|C=\phi_{-}(\tau(C))\},
\]%
\[
D_{2}:=\{C\in(-\infty,0)|C=\phi_{+}(\tau(C))\}.
\]
Since $\tau(C)$ is continuous as a function of $C$, we see that $D_{1},D_{2}$
are both closed subsets of $(-\infty,0)$, so%
\[
\text{either }D_{1}=(-\infty,0),D_{2}=\emptyset \text{ \ or \ }D_{1}%
=\emptyset,D_{2}=(-\infty,0).
\]
Without loss of generality we may assume $D_{1}=(-\infty,0)$. Then $C=\phi
_{-}(\tau(C))$ for any $C\in(-\infty,0)$. This proves that $(-\infty,0)\ni
C\mapsto \tau(C)\in \mathcal{C}_{-}$ is one-to-one. Similarly, $(1,+\infty)\ni
C\mapsto \tau(C)\in \mathcal{C}_{+}$ is one-to-one. Finally, since%
\[
C\in(0,1)\Longleftrightarrow \frac{1}{1-C}\in(1,+\infty),
\]
it follows from (\ref{i-65}) that $(0,1)\ni C\mapsto \tau(C)\in \mathcal{C}_{0}$
is also one-to-one. The proof is complete.
\end{proof}

\begin{lemma}\label{lemma33}
Let $\tau=\frac{1}{2}+ib$ with $b\geq \frac{\sqrt{3}}{2}$ and recall (\ref{i-66}) that $g_2(\tau)\geq 0$. Then
\begin{equation}
\eta_{1}(\tau)-\sqrt{g_{2}(\tau)/12}<\frac{2\pi}{b}. \label{ex}%
\end{equation}

\end{lemma}

\begin{proof}
Clearly $q=e^{2\pi i\tau}=-e^{-2\pi b}$. It follows from the $q$-expansions (\ref{ex-1}) and (\ref{ex-2}) that
\begin{equation}\label{1/2-eta}
\eta_{1}(\tau)=\frac{\pi^{2}}{3}-8\pi^{2}\sum_{k=1}^{\infty}(-1)^{k}b_{k}e^{-2k\pi b},\quad b_{k}=\sum_{1\leq d|k}d,
\end{equation}%
\begin{equation}\label{1/2-g}
g_{2}(\tau)=\frac{4}{3}\pi^{4}+320\pi^{4}\sum_{k=1}^{\infty}(-1)^{k}\sigma
_{3}(k)e^{-2k\pi b},\;\; \sigma_{3}(k)=\sum_{1\leq d|k}d^{3}.
\end{equation}
Since $b\geq \frac{\sqrt{3}}{2}$, we have $e^{\pi b}>15$. It is easy to prove
that%
\[
b_{k}<15^{\frac{k}{4}}<e^{\frac{k\pi b}{4}},\text{ \ }\sigma_{3}(k)\leq
b_{k}^{3}<e^{\frac{3k\pi b}{4}},\text{ \ }\forall k\geq1.
\]
Consequently,%
\[
\sum_{k=3}^{\infty}b_{k}e^{-2k\pi b}\leq \sum_{k=3}^{\infty}e^{-\frac{7}{4}k\pi
b}=\frac{e^{-\frac{21}{4}\pi b}}{1-e^{-\frac{7}{4}\pi b}}<e^{-4\pi b},
\]
i.e.%
\begin{align*}
  -\sum_{k=1}^{\infty}(-1)^{k}b_{k}e^{-2k\pi b}
=e^{-2\pi b}-3e^{-4\pi b}-\sum_{k=3}^{\infty}(-1)^{k}b_{k}e^{-2k\pi b}%
\in(0,e^{-2\pi b}).
\end{align*}
From here and (\ref{1/2-eta}) we obtain%
\begin{equation}
\frac{\pi^{2}}{3}<\eta_{1}(\tau)<\frac{\pi^{2}}{3}+8\pi^{2}e^{-2\pi b},\quad\forall\,b\geq\frac{\sqrt{3}}{2}.
\label{ex-3}%
\end{equation}
Together with Corollary \ref{line1/2} that $\frac{d}{db}\eta_{1}(\frac{1}%
{2}+ib)\not =0$ for $b\geq \frac{1}{2\sqrt{3}}$, we conclude that \begin{equation}\label{de-eta}\frac{d}%
{db}\eta_{1}(\tfrac{1}{2}+ib)<0\quad\text{ for}\quad b\geq \frac{1}{2\sqrt{3}},\end{equation} so (\ref{eta1rho}) gives
\[
\eta_{1}(\tfrac{1}{2}+ib)\leq \eta_{1}(e^{i\pi/3})=\frac{2\pi}{\sqrt{3}}\;\text{ for any
}b\geq \frac{\sqrt{3}}{2}.
\]
Since $g(\tau)\geq0$ for $b\geq \frac{\sqrt{3}}{2}$, we see that (\ref{ex})
holds for any $b\in \lbrack \frac{\sqrt{3}}{2},\sqrt{3})$.

Now we assume $b\geq \frac{6}{5}$ (indeed $b\geq \sqrt{3}$ is enough,
here we assume $b\geq \frac{6}{5}$ for later use). Then $e^{\pi b}>40$.
Similarly we have%
\[
\sum_{k=2}^{\infty}\sigma_{3}(k)e^{-2k\pi b}\leq \sum_{k=2}^{\infty}%
e^{-\frac{5}{4}k\pi b}=\frac{e^{-\frac{5}{2}\pi b}}{1-e^{-\frac{5}{4}\pi b}%
}<\frac{1}{2}e^{-2\pi b},
\]
i.e.%
\[
\sum_{k=1}^{\infty}(-1)^{k}\sigma_{3}(k)e^{-2k\pi b}=-e^{-2\pi b}+\sum
_{k=2}^{\infty}(-1)^{k}\sigma_{3}(k)e^{-2k\pi b}>-\frac{3}{2}e^{-2\pi b}.
\]
Thus (\ref{1/2-g}) and $e^{\pi b}>40$ give%
\[
g_{2}(\tau)>\frac{4}{3}\pi^{4}\left(  1-360e^{-2\pi b}\right)  >\frac{4}{3}%
\pi^{4}(1-210e^{-2\pi b})^{2},
\]
i.e.%
\begin{equation}
\sqrt{g_{2}(\tau)/12}>\frac{\pi^{2}}{3}(1-210e^{-2\pi b}),\text{ \  \ }\forall
b\geq \frac{6}{5}. \label{ex-6}%
\end{equation}
Together with (\ref{ex-3}), we obtain%
\[
\eta_{1}(\tau)-\sqrt{g_{2}(\tau)/12}<78\pi^{2}e^{-2\pi b},\text{ \  \ }\forall
b\geq \frac{6}{5}.
\]
Since it is trivial to see that $78\pi^{2}e^{-2\pi b}<\frac{2\pi}{b}$ for
$b\geq \sqrt{3}$, we conclude that (\ref{ex}) holds for any $b\geq \sqrt{3}$.
This completes the proof.
\end{proof}

\begin{proof}
[Proof of Theorem \ref{lemma-15}](i) is just Lemmas \ref{lemma-11} and \ref{lemma-14}.

(ii). We will prove in Theorem \ref{thm-12} below that $\eta_{1}(1-\bar{\tau
})=\overline{\eta_{1}(\tau)}$ and $\eta_{2}(1-\bar{\tau})=\overline{\eta
_{1}(\tau)}-\overline{\eta_{2}(\tau)}$. By the $q$-expansion (\ref{ex-2}) we also have $g_{2}(1-\bar{\tau})=\overline{g_{2}(\tau)}$.  Since $C\in \mathbb{R}\backslash
\{0,1\}$, we easily obtain%
\begin{align*}
  f_{1-C}(1-\bar{\tau})
=&12\left(  (1-C)\eta_{1}(1-\bar{\tau})-\eta_{2}(1-\bar{\tau})\right)
^{2}-g_{2}(1-\bar{\tau})(C-\bar{\tau})^{2}\\
=&12(C\overline{\eta_{1}(\tau)}-\overline{\eta_{2}(\tau)})^{2}%
-\overline{g_{2}(\tau)}(C-\bar{\tau})^{2}=\overline{f_{C}(\tau)}.
\end{align*}
Therefore, it follows from Theorem \ref{Unique-pole}-(1) that
\begin{equation}
\tau(1-C)=1-\overline{\tau(C)}. \label{i-67}%
\end{equation}
Since $\tau$ and $1-\bar{\tau}$ is symmetric with respect to the line
$\operatorname{Re}\tau=\frac{1}{2}$, we see that assertion (ii) holds.

(iii). By (ii), $\mathcal{C}_{0}$ has intersections with the line
$\operatorname{Re}\tau=\frac{1}{2}$. Let $\tau_{0}=\frac{1}{2}+ib$ be such an
intersection point. Then $\tau_{0}=\tau(C)$ for a unique $C\in(0,1)$. Applying
(\ref{i-67}), we have $\tau(1-C)=\tau_{0}=\tau(C)$, so Lemma \ref{lemma-14}
gives $1-C=C$, i.e. $C=\frac{1}{2}$. This proves that $\tau_{0}=\tau(\frac
{1}{2})$ is the unique intersection point of the curve $\mathcal{C}_{0}$ with
the line $\operatorname{Re}\tau=\frac{1}{2}$. By (\ref{phipm}),%
\begin{align*}
\frac{1}{2}  &  =\phi_{\pm}(\tau_{0})=\tau_{0}-\frac{2\pi i}{\eta_{1}(\tau
_{0})\pm \sqrt{g_{2}(\tau_{0})/12}}\\
&  =\frac{1}{2}+ib-\frac{2\pi i}{\eta_{1}(\tau_{0})\pm \sqrt{g_{2}(\tau
_{0})/12}},
\end{align*}
we see that $\eta_{1}(\tau_{0})\pm \sqrt{g_{2}(\tau_{0})/12}=\frac{2\pi}{b}$ is
real, so $g_{2}(\tau_{0})\geq0$, i.e. $b=\operatorname{Im}\tau_{0}\geq
\frac{\sqrt{3}}{2}$. Since (\ref{eta1rho}) gives $\eta_{1}(e^{\pi i/3})=\frac{2\pi}{\sqrt{3}}%
\not =\frac{4\pi}{\sqrt{3}}$, we obtain $b=\operatorname{Im}\tau_{0}%
>\frac{\sqrt{3}}{2}$. Then $\operatorname{Im}\frac{\tau(1/2)-1}{2\tau(1/2)-1}<\frac{1}{2\sqrt{3}}$ and so Lemma \ref{lemma33} applies. In particular, (\ref{ex}) infers
\begin{equation}
\eta_{1}(\tau_{0})+\sqrt{g_{2}(\tau_{0})/12}=\frac{2\pi}{b}. \label{ex-7}%
\end{equation}
Suppose $b=\operatorname{Im}\tau_{0}\geq \frac{6}{5}$. Then (\ref{ex-3}),
(\ref{ex-6}) and (\ref{ex-7}) imply%
\[
\frac{2\pi^{2}}{3}-70\pi^{2}e^{-2\pi b}<\frac{2\pi}{b}\leq \frac{5\pi}{3},
\]
which is equivalent to $\frac{210\pi}{2\pi-5}>e^{2\pi b}$, clearly a
contradiction with $e^{\pi b}\geq e^{6\pi/5}>40$. Thus $b=\operatorname{Im}%
\tau_{0}<\frac{6}{5}$. This proves (iii).

(iv)-(v). Note from (\ref{i8}) that if $\mathcal{C}_{-}$ (resp. $\mathcal{C}%
_{+}$) has a intersection point $\tau$ with the line $\operatorname{Re}%
\tau=\frac{1}{2}$, then $\operatorname{Im}\tau \in \lbrack \frac{1}{2}%
,\frac{\sqrt{3}}{2})$. Assume $\tau=\frac{1}{2}+ib$ with $b\in \lbrack \frac
{1}{2},\frac{\sqrt{3}}{2})$ is a intersection point of $\mathcal{C}_{-}$ with
the line $\operatorname{Re}\tau=\frac{1}{2}$. Then (\ref{i-66}) gives
$\sqrt{g_{2}(\tau)/12}=\pm i\sqrt{|g_{2}(\tau)|/12}$. Clearly there exists
$C\in(-\infty,0)$ such that $\tau=\tau(C)$, which implies $C=\phi_{\pm}(\tau
)$, i.e.
{\allowdisplaybreaks
\begin{align*}
C  &  =\tau-\frac{2\pi i}{\eta_{1}(\tau)\pm i\sqrt{|g_{2}(\tau)|/12}}\\
&  =\frac{1}{2}+bi-\frac{2\pi i(\eta_{1}(\tau)\mp i\sqrt{|g_{2}(\tau)|/12}%
)}{\eta_{1}(\tau)^{2}-\frac{g_{2}(\tau)}{12}}.
\end{align*}
}%
Thus
\begin{equation}
b-\frac{2\pi \eta_{1}(\tau)}{\eta_{1}(\tau)^{2}-\frac{g_{2}(\tau)}{12}}=0.
\label{i9}%
\end{equation}
For $\tau=\frac{1}{2}+ib$ with $b\in \lbrack \frac{1}{2},\frac{\sqrt{3}}{2}]$,
we define%
\[
\theta(b):=\frac{b\eta_{1}(\frac{1}{2}+ib)}{2\pi}\;\text{ and }\;\theta
_{1}(b):=\frac{\eta_{1}(\frac{1}{2}+ib)^{2}}{\eta_{1}(\frac{1}{2}%
+ib)^{2}-\frac{g_{2}(\frac{1}{2}+ib)}{12}}.
\]
Then (\ref{i9}) shows that if $\tau=\frac{1}{2}+ib$ with $b\in \lbrack \frac
{1}{2},\frac{\sqrt{3}}{2})$ is a intersection point of $\mathcal{C}_{-}$ with
the line $\operatorname{Re}\tau=\frac{1}{2}$, then%
\begin{equation}
\theta(b)-\theta_{1}(b)=0. \label{i10}%
\end{equation}

Now we want to prove that (\ref{i10}) has a unique solution in $[\frac{1}%
{2},\frac{\sqrt{3}}{2}]$. For this, first we note from (\ref{de-eta}) that
\begin{equation}
\eta_{1}(\tfrac{1}{2}+bi)>\pi^2/3\quad\text{ for }\quad b\geq \tfrac{1}{2\sqrt{3}}. \label{i22}%
\end{equation}%
We also recall the following fundamental
results (cf. \cite{LW}):%
\begin{equation}
e_{1}(\tfrac{1}{2}+\tfrac{1}{2}i)=0\text{ and }e_{1}(\tfrac{1}{2}+bi)>0\text{
for }b>\tfrac{1}{2}, \label{i22-1}%
\end{equation}%
\begin{equation}
g_{3}(\tfrac{1}{2}+bi)=4e_{1}(\tfrac{1}{2}+bi)|e_{2}(\tfrac{1}{2}%
+bi)|^{2}>0\text{ for }b\in(\tfrac{1}{2},\tfrac{\sqrt{3}}{2}], \label{i24}%
\end{equation}%
\begin{equation}
g_{2}(\tfrac{1}{2}+bi)^{3}-27g_{3}(\tfrac{1}{2}+bi)^{2}<0\text{ for }%
b\geq \tfrac{1}{2}. \label{i23}%
\end{equation}

By using (see e.g. \cite[p.704]{YB} or \cite[Appendix B]{CKL2})%
\begin{equation}
\frac{d}{db}\eta_{1}(\tfrac{1}{2}+bi)=i\eta_{1}^{\prime}(\tau)=\frac{-1}{4\pi
}(2\eta_{1}^{2}-\tfrac{1}{6}g_{2}), \label{i30}%
\end{equation}%
\[
\frac{d}{db}g_{2}(\tfrac{1}{2}+bi)=ig_{2}^{\prime}(\tau)=\frac{1}{\pi}%
(3g_{3}-2\eta_{1}g_{2}),
\]
\begin{equation}
\frac{d}{db}g_{3}(\tfrac{1}{2}+bi)=ig_{3}^{\prime}(\tau)=\frac{1}{\pi}%
(-3g_{3}\eta_{1}+\tfrac{1}{6}g_{2}^{2}), \label{i31}%
\end{equation}
we easily obtain%
\begin{align}
\theta^{\prime}(b)    =\frac{\eta_{1}}{2\pi}-\frac{b}{8\pi^{2}}(2\eta_{1}^{2}-\tfrac{1}{6}g_{2})
  =\frac{\eta_{1}(\frac{1}{2}+ib)}{2\pi \theta_{1}(b)}(\theta_{1}
(b)-\theta(b)),\label{i26}
\end{align}%
{\allowdisplaybreaks
\begin{align}
\theta_{1}^{\prime}(b)  &  =\frac{1}{12}\frac{d}{db}\bigg(  \frac{g_{2}%
(\frac{1}{2}+ib)}{\eta_{1}(\frac{1}{2}+ib)^{2}-\frac{g_{2}(\frac{1}{2}%
+ib)}{12}}\bigg) \label{i27}\\
&  =\frac{\eta_{1}}{12\pi(\eta_{1}^{2}-\frac{g_{2}}{12})^{2}}\left(
3g_{3}\eta_{1}-\eta_{1}^{2}g_{2}-\tfrac{1}{12}g_{2}^{2}\right)  .\nonumber
\end{align}
}%

\textbf{Step 1.} We claim that $\theta^{\prime}(\tfrac{1}{2})<0$.

Indeed, by (\ref{eta1i}) and $\eta_2(\tau)=\tau\eta_1(\tau)-2\pi i$ we obtain \[\eta_{1}(i)=\pi\quad\text{  and  }\quad \eta_{2}(i)=-\pi i,\]
which imply%
\[
\eta_{1}(\tfrac{1}{2}+\tfrac{1}{2}i)=\eta_{1}(\tfrac{1}{1-i})=(1-i)(\eta
_{1}(i)-\eta_{2}(i))=2\pi.
\]
Recalling Lemma \ref{lemma-etai} that $g_{2}(i)>12\eta_{1}(i)^{2}=12\pi^{2}$, we have%
\[
g_{2}(\tfrac{1}{2}+\tfrac{1}{2}i)=g_{2}(\tfrac{1}{1-i})=(1-i)^{4}%
g_{2}(i)=-4g_{2}(i)<-48\pi^{2},
\]
so%
\[
0<\theta_{1}(\tfrac{1}{2})=\frac{\eta_{1}(\tfrac{1}{2}+\tfrac{1}{2}i)^{2}}%
{\eta_{1}(\tfrac{1}{2}+\tfrac{1}{2}i)^{2}-\frac{g_{2}(\tfrac{1}{2}+\tfrac
{1}{2}i)}{12}}<\tfrac{1}{2}=\theta(\tfrac{1}{2}).
\]
This, together with (\ref{i26}), proves
$\theta^{\prime}(\tfrac{1}{2})<0$.\medskip

\textbf{Step 2. }We claim that for any $b\in(\frac{1}{2},\frac{\sqrt{3}}{2})$
satisfying $\theta_{1}^{\prime}(b)=0$, there holds $\psi^{\prime}(b)>0$, where%
\[
\psi(b):=(3g_{3}\eta_{1}-\eta_{1}^{2}g_{2}-\tfrac{1}{12}g_{2}^{2})(\tfrac
{1}{2}+bi).
\]

Applying (\ref{i30})-(\ref{i31}), a straightforward computation leads to%
\[
\psi^{\prime}(b)=\frac{1}{\pi}\left(  3g_{2}\eta_{1}^{3}-\tfrac{27}{2}%
g_{3}\eta_{1}^{2}+\tfrac{3}{4}g_{2}^{2}\eta_{1}-\tfrac{3}{8}g_{2}g_{3}\right)
.
\]
By (\ref{i22}), (\ref{i27}) and $\theta_{1}^{\prime}(b)=0$, we have $\psi(b)=0$, i.e.%
\[
\eta_{1}^{2}g_{2}=3g_{3}\eta_{1}-\tfrac{1}{12}g_{2}^{2}.
\]
Inserting this into the term $3g_{2}\eta_{1}^{3}$ of $\psi^{\prime}(b)$, we
easily obtain%
\[
\psi^{\prime}(b)=\frac{1}{\pi}\left(  -\tfrac{9}{2}g_{3}\eta_{1}^{2}+\tfrac
{1}{2}g_{2}^{2}\eta_{1}-\tfrac{3}{8}g_{2}g_{3}\right)  ,
\]
and so%
\begin{align*}
g_{2}\psi^{\prime}(b)  &  =\frac{1}{\pi}\left(  -\tfrac{9}{2}g_{3}(3g_{3}%
\eta_{1}-\tfrac{1}{12}g_{2}^{2})+\tfrac{1}{2}g_{2}^{3}\eta_{1}-\tfrac{3}%
{8}g_{2}^{2}g_{3}\right) \\
&  =\frac{\eta_{1}}{2\pi}(g_{2}^{3}-27g_{3}^{2}).
\end{align*}
Applying (\ref{i-66}), (\ref{i22}) and (\ref{i23}) we conclude $\psi^{\prime}(b)>0$.\medskip

\textbf{Step 3. }As a direct consequence of Step 2 and (\ref{i27}), we have
that if $\theta_{1}^{\prime}(b_{0})\geq0$ for some $b_{0}\in(\frac{1}{2}%
,\frac{\sqrt{3}}{2})$, then $\theta_{1}^{\prime}(b)>0$ for any $b\in
(b_{0},\frac{\sqrt{3}}{2})$.\medskip

\textbf{Step 4. }We claim that there exists $b_{1}\in(\frac{1}{2},\frac
{\sqrt{3}}{2})$ such that $\theta^{\prime}(b)<0$ for $b\in \lbrack \frac{1}%
{2},b_{1})$, $\theta^{\prime}(b_{1})=0$ and $\theta^{\prime}(b)>0$ for
$b-b_{1}>0$ small. Consequently,%
\begin{equation}
\theta_{1}^{\prime}(b_{1})\geq0\text{ and }\theta_{1}^{\prime}(b)>0\text{ for
any }b\in(b_{1},\tfrac{\sqrt{3}}{2}). \label{i37}%
\end{equation}

Since $\eta_{1}(\frac{1}{2}+\frac{\sqrt{3}}{2}i)=\frac{2\pi}{\sqrt{3}}$ gives
$\theta(\frac{\sqrt{3}}{2})=\frac{1}{2}=\theta(\frac{1}{2})$, Step 1 implies
that there exists $b_{1}\in(\frac{1}{2},\frac{\sqrt{3}}{2})$ such that
$\theta^{\prime}(b)<0$ for $b\in \lbrack \frac{1}{2},b_{1})$ and $\theta
^{\prime}(b_{1})=0$. Suppose $\theta^{\prime}(b)\leq0$ for $b-b_{1}>0$
sufficiently small, then (\ref{i26}) says that $\theta_{1}(b)-\theta(b)$ has a
local maximum $0$ at $b_{1}$, so $\theta_{1}^{\prime}(b_{1})=0$ and then Step
3 gives $\theta_{1}^{\prime}(b)>0$ for $b-b_{1}>0$ small. However, this
implies $\theta_{1}^{\prime}(b)-\theta^{\prime}(b)>0$ for $b-b_{1}>0$
sufficiently small, which contradicts with that $\theta_{1}(b)-\theta(b)$ has
a local maximum $0$ at $b_{1}$. Therefore, $\theta^{\prime}(b)>0$ for
$b-b_{1}>0$ small. This also gives $\theta_{1}(b)-\theta(b)>0$ for $b-b_{1}>0$
small. Since $\theta_{1}(b_{1})-\theta(b_{1})=0$, we obtain $\theta
_{1}^{\prime}(b_{1})=\theta_{1}^{\prime}(b_{1})-\theta^{\prime}(b_{1})\geq0$ and so (\ref{i37}) holds.\medskip

\textbf{Step 5. }We claim that $\theta^{\prime}(b)>0$ for $b\in(b_{1}%
,\frac{\sqrt{3}}{2}]$.

Since $g_{2}(\frac{1}{2}+\frac{\sqrt{3}}{2}i)=0$ gives $\theta_{1}(\frac
{\sqrt{3}}{2})=1>\theta(\frac{\sqrt{3}}{2})$, we have $\theta^{\prime}%
(\frac{\sqrt{3}}{2})>0$. Suppose there exists $b_{2}\in(b_{1},\frac{\sqrt{3}%
}{2})$ such that $\theta^{\prime}(b_{2})=0$ and $\theta^{\prime}(b)>0$ for
$b\in(b_{1},b_{2})$. Then (\ref{i26}) gives $\theta_{1}(b)-\theta(b)>0$ for
$b\in(b_{1},b_{2})$ and $\theta_{1}(b_{2})-\theta(b_{2})=0$, i.e. $\theta
_{1}^{\prime}(b_{2})-\theta^{\prime}(b_{2})\leq0$. However, by (\ref{i37}) we
have $\theta_{1}^{\prime}(b_{2})-\theta^{\prime}(b_{2})=\theta_{1}^{\prime}(b_{2})>0$, a
contradiction.\medskip

\textbf{Step 6. }We finish the proof of assertions (iv)-(v).

By Steps 4-5 and (\ref{i26}), $\theta_{1}(b)-\theta(b)=0$ has a unique
solution $b_{1}$ in $[\frac{1}{2},\frac{\sqrt{3}}{2}]$ and $b_1 \in(\frac{1}{2},\frac
{\sqrt{3}}{2})$. Together with
(\ref{i10}), we see that $\mathcal{C}_{-}$ has at most one intersection point
$\frac{1}{2}+ib_{1}$ with the line $\operatorname{Re}\tau=\frac{1}{2}$. Since
Lemma \ref{lemma-11} implies that $\mathcal{C}_{-}$ must intersect with the
line $\operatorname{Re}\tau=\frac{1}{2}$, we conclude that $\tau_{-}:=\frac
{1}{2}+ib_{1}$ is the unique intersection point of $\mathcal{C}_{-}$ with the
line $\operatorname{Re}\tau=\frac{1}{2}$. Finally, it follows from assertion
(ii) that $\tau_{-}$ is also the unique intersection point of $\mathcal{C}%
_{+}$ with the line $\operatorname{Re}\tau=\frac{1}{2}$ (resp. $\mathcal{C}%
_{-}$).

(vi). Clearly this follows readily from the assertion (v) and (\ref{i-65}).
The proof is complete.
\end{proof}

\section{Geometric interpretation and smoothness of the curves}

\label{geometricinter}
The purpose of this section is to give the geometric meaning of the three curves from the multiple Green function $G_{2}$.
As pointed out in Section 1, $\{(q_{\pm},-q_{\pm})|\wp(q_{\pm})=\pm \sqrt{g_{2}/12}\}$ are trivial critical points of $G_{2}$, and it was calculated in \cite[Example 4.2]{LW3} that the
Hessian is given by
\begin{equation}
\det D^{2}G_{2}(q_{\pm},-q_{\pm};\tau)=\frac{3|g_{2}(\tau)|}{4\pi
^{4}\operatorname{Im}\tau}|\wp(q_{\pm}|\tau)+\eta_{1}(\tau)|^{2}\operatorname{Im}%
\phi_{\pm}(\tau), \label{SZ-2}%
\end{equation}
where $\phi_{\pm}(\tau)$ is defined in (\ref{i-53}). Define
the following \emph{degeneracy curve} of $G_{2}$ on $F_{0}$ (i.e. consist of those $\tau$'s such that either $(q_+,-q_+)$ or $(q_-,-q_-)$ is a degenerate critical point of $G_2(\cdot,\cdot;\tau)$):\footnote{ $g_{2}(e^{\pi i/3})=0$ implies that the two trivial critical points $(q_{\pm},-q_{\pm})$ degenerate to one trivial critical point at $\tau=e^{\pi i/3}$, so we exclude the trivial case in the definition of $C_{+,-}$.}
\[C_{+,-}:=\left \{  \tau\in F_{0}\setminus\{e^{\pi i/3}\}\left \vert
\begin{array}
[c]{r}%
\det D^{2}G_{2}(q_{+},-q_{+};\tau)=0\\
\text{or }\det D^{2}G_{2}(q_{-},-q_{-};\tau)=0
\end{array}
\right.  \right \}.
\]

\begin{theorem}
\label{degeneracy} The degeneracy curve $C_{+,-}=\mathcal{C}_{-}\cup\mathcal{C}_{0}\cup\mathcal{C}_{+}$. In particular, $\mathcal{C}_{-}$, $\mathcal{C}_{0}$, $\mathcal{C}_{+}$ are all smooth curves.
\end{theorem}

\begin{proof}
Remark that  $\wp(q_{\pm}|\tau)+\eta_{1}(\tau)=\eta_{1}(\tau)\pm\sqrt{g_2(\tau)/12}=0$ implies $12\eta_{1}(\tau)^{2}-g_{2}(\tau)=0$. So it follows from  (\ref{eta1g2}) that $\eta_{1}(\tau)\pm\sqrt{g_2(\tau)/12}\neq 0$ for any $\tau\in F_{0}$. Consequently, we deduce from (\ref{SZ-2}) that
{\allowdisplaybreaks
\begin{align}\label{c+-}
C_{+,-}&=\left \{  \tau\in F_{0}\setminus\{e^{\pi i/3}\}\left \vert
\operatorname{Im}\phi_{+}(\tau)=0
\;\text{or }\operatorname{Im}\phi_{-}(\tau)=0
\right.  \right \}\\
&=\left \{  \tau\in F_{0}\setminus\{e^{\pi i/3}\}\left \vert
\begin{array}
[c]{r}%
\phi_{+}(\tau)=C\text{ or }\phi_{-}(\tau)=C\nonumber\\
\text{ for some }C\in\mathbb{R}
\end{array}
\right.  \right \}\nonumber\\
&=\left \{  \tau\in F_{0}\setminus\{e^{\pi i/3}\}\left \vert f_{C}(\tau)=0 \text{ for some }C\in\mathbb{R}\right.  \right \}\nonumber\\
&=\left \{  \tau(C)| C\in\mathbb{R}\setminus\{0,1\} \right \}
=\mathcal{C}_{-}\cup\mathcal{C}_{0}\cup\mathcal{C}_{+},\nonumber
\end{align}
}%
where we have used Theorem \ref{Unique-pole} and Theorem \ref{lemma-15} (iii)-(iv), which show that $e^{\pi i/3}\notin \mathcal{C}_{-}\cup\mathcal{C}_{0}\cup\mathcal{C}_{+}$.

Now we prove that the three curves are all smooth curves in $F_{0}$. Recalling (\ref{i8}) in Lemma \ref{lemma-14}, we restrict $\tau \in F_{0}\backslash \{ \tau=\tfrac{1}{2}+ib|
b\geq \tfrac{\sqrt{3}}{2}\}$ and fix a branch of $\sqrt{g_{2}(\tau)/12}$. Then
both $\phi_{+}(\tau)$ and $\phi_{-}(\tau)$ are \emph{single-valued} in $F_{0}\backslash \{ \tau=\tfrac{1}{2}+ib|
b\geq \tfrac{\sqrt{3}}{2}\}$ and it follows from \cite[Theorem 3.1]{CKL2} that \begin{equation}\label{dephi}\phi_{+}'(\tau)\neq 0,\quad\phi_{-}'(\tau)\neq 0,\quad\forall \tau\in F_{0}\backslash \{ \tau=\tfrac{1}{2}+ib|
b\geq \tfrac{\sqrt{3}}{2}\}.\end{equation} By (\ref{c+-}), the same argument as Lemma \ref{lemma-14} implies that either
\begin{equation}\label{phi-}\mathcal{C}_{-} \subset \{\tau\in F_{0}\backslash \{ \tau=\tfrac{1}{2}+ib|
b\geq \tfrac{\sqrt{3}}{2}\}|\operatorname{Im}\phi_{+}(\tau)=0\}\end{equation} or
\[\mathcal{C}_{-} \subset \{\tau\in F_{0}\backslash \{ \tau=\tfrac{1}{2}+ib|
b\geq \tfrac{\sqrt{3}}{2}\}|\operatorname{Im}\phi_{-}(\tau)=0\}.\]
Say (\ref{phi-}) holds for example. Write $\tau=a+bi$ with $a,b\in\mathbb{R}$. By (\ref{phi-}), (\ref{dephi}) and \[
\frac{\partial \operatorname{Im}\phi_{+}}{\partial a}=\operatorname{Im}\phi
_{+}^{\prime},\text{ \  \ }\frac{\partial \operatorname{Im}\phi_{+}}{\partial
b}=\operatorname{Re}\phi_{+}^{\prime},
\]
we see that $\mathcal{C}_{-}$ is smooth at any $\tau\in \mathcal{C}_{-}$. Therefore,
in both cases, we can apply (\ref{dephi}) to conclude that $\mathcal{C}_{-}$ is a smooth curve. Then Theorem \ref{lemma-15}-(ii) shows that $\mathcal{C}_{+}$ is also a smooth curve. For $\mathcal{C}_{0}$, we note from Theorem \ref{lemma-15}-(iii) that $\mathcal{C}_{0}\subset F_{0}\backslash \{ \tau=\tfrac{1}{2}+ib|
b\leq \tfrac{\sqrt{3}}{2}\}$. Restrict $\tau \in F_{0}\backslash \{ \tau=\tfrac{1}{2}+ib|
b\leq \tfrac{\sqrt{3}}{2}\}$ and fix a branch of $\sqrt{g_{2}(\tau)/12}$. Again
both $\phi_{+}(\tau)$ and $\phi_{-}(\tau)$ are \emph{single-valued} in $F_{0}\backslash \{ \tau=\tfrac{1}{2}+ib|
b\leq \tfrac{\sqrt{3}}{2}\}$, so the same argument shows that $\mathcal{C}_{0}$ is a smooth curve.
\end{proof}

The numerical simulation for the degeneracy curves of $G_2$ and hence the three smooth curves is shown in Figure 3, which is copied from C. L. Wang \cite{LW2}. The other six curves appearing in Figure 3 are those degeneracy curves of $G_2$ at other trivial critical points $\{(\frac{1}{2}\omega_i, \frac{1}{2}\omega_j)|i\neq j\}$. In another paper, we will prove that under the $\Gamma_{0}(2)$ action, all critical points of $e_k(\tau)=\wp(\frac{\omega_k}{2}|\tau)$ will be mapped to locating on these six smooth curves. Figure 3 indicates that critical points of $E_{2}(\tau)$ and $e_k(\tau)$'s could be approximately computed via mathematical softwares such as Mathematica.

\begin{figure}[btp] \label{deg-curve-2}
\includegraphics[width=2.4in]{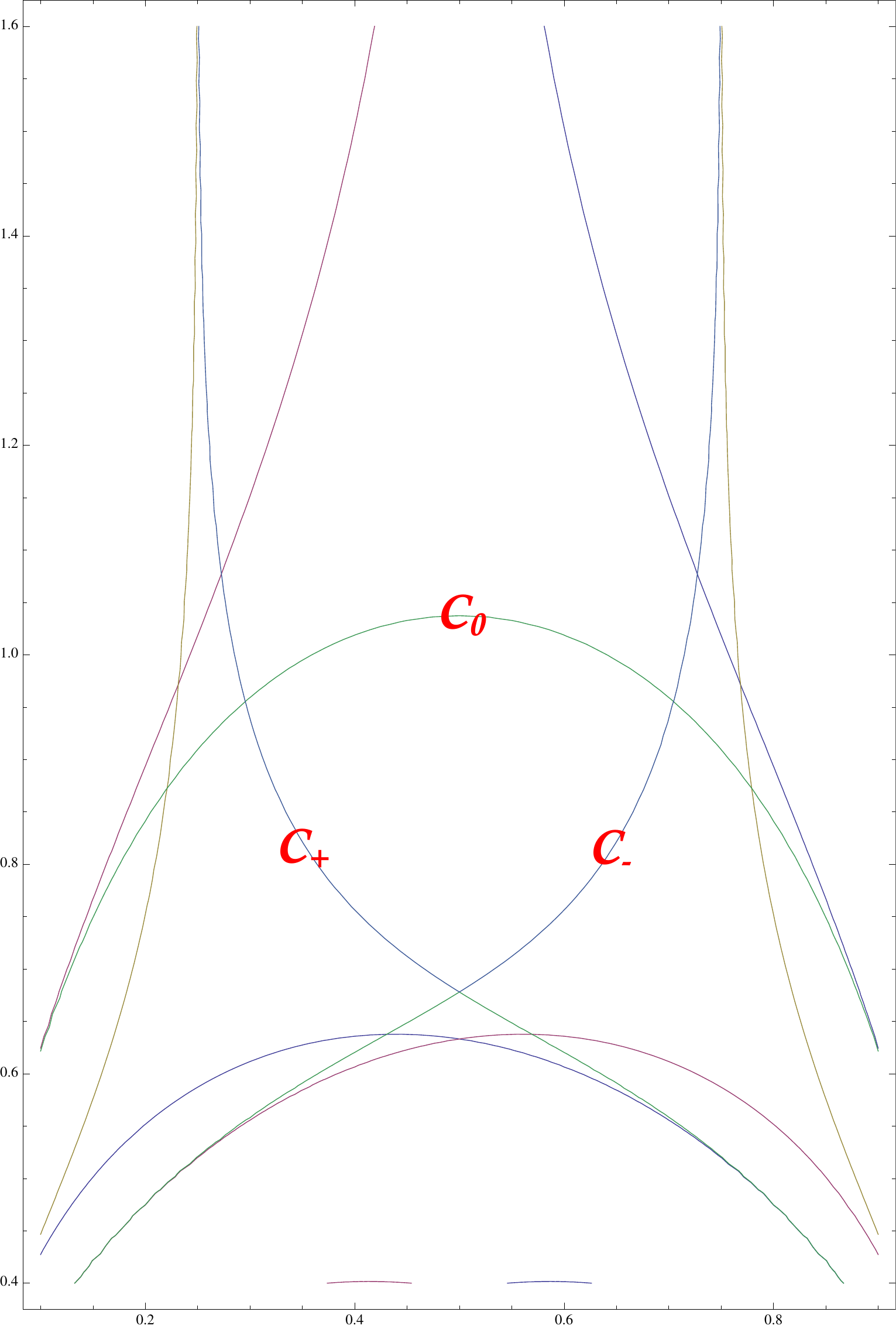}
\caption{The smooth curves.}
\end{figure}

\appendix

\section{Application}

In this appendix,
we apply Theorem \ref{thm2} back to the mean field equation (\ref{mfe}).
Define%
\[
L:=\{(r,s)\in \triangle_{1}\text{ }|\text{ }2r+s=2\}.
\]
By Theorem \ref{thm2}, for any $(r,s)\in L$, $Z_{r,s}^{(2)}(\tau)$ has a
unique zero, denoted by $\tau_{s}$, in $F_{0}$.

\begin{theorem}
\label{thm-12}For any $(r,s)\in L$, the unique zero $\tau_{s}$ of $Z_{r,s}^{(2)}(\tau)$ in $F_0$ satisfies
 $\tau_{s}=\frac{1}{2}+ib_{s}$ for some $b_{s}>\frac{1}{2}$. Furthermore,
$\lim_{s\rightarrow \frac{1}{2}}b_{s}=+\infty$ and%
\begin{equation}
b^{\ast}:=\lim_{s\rightarrow0}b_{s}\text{ exists and }\; b^{\ast}\in(\sqrt
{3}/2,6/5).\label{sss}%
\end{equation}
In particular, for any $\tau=\frac{1}{2}+ib$ with $b\in(b^{\ast},+\infty)$,
there exists $(r,s)\in L$ such that this $\tau$ is the unique zero of
$Z_{r,s}^{(2)}$ in $F_{0}$.
\end{theorem}

\begin{proof}
Given any $(r,s)\in L$. By the definition of $\zeta(z|\tau)$ and $\wp(z|\tau
)$, it is easy to prove%
\[
\overline{\zeta(z|\tau)}=\zeta(\bar{z}|1-\bar{\tau}),\text{ }\overline
{\wp(z|\tau)}=\wp(\bar{z}|1-\bar{\tau}),
\]%
\[
\overline{\wp^{\prime}(z|\tau)}=\wp^{\prime}(\bar{z}|1-\bar{\tau}).
\]
Thus%
\[
\overline{\eta_{1}(\tau)}=2\overline{\zeta(1/2|\tau)}=2\zeta(1/2|1-\bar{\tau
})=\eta_{1}(1-\bar{\tau}),
\]%
{\allowdisplaybreaks
\begin{align*}
\overline{\eta_{2}(\tau)}  &  =2\overline{\zeta(\tau/2|\tau)}=2\zeta(\bar
{\tau}/2|1-\bar{\tau})\\
&  =2\zeta(1/2|1-\bar{\tau})-2\zeta((1-\bar{\tau})/2|1-\bar{\tau})\\
&  =\eta_{1}(1-\bar{\tau})-\eta_{2}(1-\bar{\tau}),
\end{align*}
}%
i.e.
{\allowdisplaybreaks
\begin{align*}
\overline{Z_{r,s}(\tau)}  &  =\overline{\zeta(r+s\tau|\tau)}-r\overline
{\eta_{1}(\tau)}-s\overline{\eta_{2}(\tau)}\\
&  =\zeta(r+s-s(1-\bar{\tau})|1-\bar{\tau})-(r+s)\eta_{1}(1-\bar{\tau}%
)+s\eta_{2}(1-\bar{\tau})\\
&  =Z_{r+s,-s}(1-\bar{\tau}).
\end{align*}
}%
From here and%
\[
\overline{\wp(r+s\tau|\tau)}=\wp(r+s-s(1-\bar{\tau})|1-\bar{\tau}),
\]%
\[
\overline{\wp^{\prime}(r+s\tau|\tau)}=\wp^{\prime}(r+s-s(1-\bar{\tau}%
)|1-\bar{\tau}),
\]
we obtain%
\begin{align*}
\overline{Z_{r,s}^{(2)}(\tau)}  &  =Z_{r+s,-s}^{(2)}(1-\bar{\tau
})=-Z_{-(r+s),s}^{(2)}(1-\bar{\tau})\\
&  =-Z_{2-r-s,s}^{(2)}(1-\bar{\tau})=-Z_{r,s}^{(2)}(1-\bar{\tau}).
\end{align*}
Then $1-\overline{\tau_{s}}$ is also a zero of $Z_{r,s}^{(2)}(\tau)$ in
$F_{0}$, so $1-\overline{\tau_{s}}=\tau_{s}$, i.e. $\tau_{s}=\frac{1}%
{2}+ib_{s}$ for some $b_{s}\in(\frac{1}{2},+\infty)$. Suppose by contradiction
that, up to a sequence,%
\[
\lim_{s\rightarrow 1/2}b_{s}=b\in \lbrack1/2,+\infty).
\]
Then $\frac{1}{2}+ib$ is a zero of $Z_{\frac{3}{4},\frac{1}{2}}^{(2)}(\tau)$
in $F_{0}$, which is a contradiction with Theorem \ref{thm2} because
$(\frac{3}{4},\frac{1}{2})\in \partial \triangle_{1}$. This proves
$\lim_{s\rightarrow \frac{1}{2}}b_{s}=+\infty$. Note $Z_{r,s}^{(2)}(\tau)=Z_{1-\frac{1}{2}s,s}^{(2)}(\tau)=Z_{-\frac{1}{2}s,s}^{(2)}(\tau)$. To prove (\ref{sss}), we recall that
the proof of Theorem \ref{Unique-pole}-(1) shows that $\lim_{s\to 0}\tau_{s}$ exists and
\[\lim_{s\to 0}\tau_{s}=\tau(\tfrac{1}{2}),\]
where $\tau(\frac{1}{2})$ is the unique zero of $f_{\frac{1}{2}}(\tau)$ in $F_{0}$. Together with Theorem \ref{lemma-15}-(iii), we obtain
$\lim_{s\to 0}b_{s}=\operatorname{Im}\tau(\tfrac{1}{2})\in(\frac{\sqrt{3}}{2}%
,\frac{6}{5})$. This proves (\ref{sss}) and hence completes the proof.
\end{proof}

The following result, which was announced in \cite{CKL3}, give new existence results for the mean field equation (\ref{mfe}) when $E_\tau$ is a rhombus torus.

\begin{theorem}
\label{thm3}Let $\tau=\frac{1}{2}+ib$ with $b>b^{\ast}$, where $b^{\ast}\in(\frac{\sqrt{3}}{2},\frac{6}%
{5})$ is in Theorem \ref{thm-12}. Then equation (\ref{mfe}) on $E_\tau$ has a solution.
\end{theorem}

\begin{proof}
This theorem is an immediate consequence of Theorem
\ref{thm-12} and Theorem B-(1).\end{proof}

Remark that Theorem \ref{thm3} is almost optimal in the sense of Theorem B-(2), which says that if
$\tau=\frac{1}{2}+\frac{\sqrt{3}}{2}i$, then equation (\ref{mfe})
on $E_{\tau}$ has \emph{no} solutions.

\medskip
\noindent{\bf Acknowledgements} The authors thank Prof. Chin-Lung Wang very much for providing the file of Figure 3 to us.

\end{document}